\documentclass[11pt]{amsart}

\usepackage[T1]{fontenc}
\usepackage[utf8]{inputenc}
\usepackage{lmodern}
\usepackage{amsmath,amssymb,amsthm,mathtools}
\usepackage{mathrsfs}
\usepackage{tikz}
\usetikzlibrary{arrows.meta}
\usepackage{float}
\usepackage{microtype}
\usepackage[a4paper,margin=32mm]{geometry}
\usepackage[hidelinks]{hyperref}
\usepackage[nameinlink,noabbrev]{cleveref}

\newcommand{\Sym}{\mathfrak S}
\newcommand{\CC}{\mathbb C}

\newcommand{\scalar}[2]{\left\langle #1,#2\right\rangle}

\theoremstyle{plain}
\newtheorem{theorem}{Theorem}[section]
\newtheorem{proposition}[theorem]{Proposition}
\newtheorem{lemma}[theorem]{Lemma}

\newtheorem{conjecture}[theorem]{Conjecture}
\crefname{conjecture}{conjecture}{conjectures}
\Crefname{conjecture}{Conjecture}{Conjectures}

\theoremstyle{definition}

\newtheorem{example}[theorem]{Example}

\theoremstyle{remark}
\newtheorem{remark}[theorem]{Remark}

\title{Stable Symmetric Series, Differential Operators, and Jack Deformations}

\author[J.-Y. Thibon]{Jean-Yves Thibon}
\address{Laboratoire d’Informatique Gaspard-Monge,
Université Gustave Eiffel,
CNRS, ESIEE Paris,
F-77454 Marne-la-Vallée, France}
\email[Jean-Yves Thibon]{jean-yves.thibon@univ-eiffel.fr} 


\date{}

\begin{document}

\begin{abstract}
We introduce stable symmetric series which encode normalized
conjugacy classes and their multiplication operators simultaneously
for all symmetric groups.  This gives a direct route from the
Ivanov--Kerov algebra to shifted symmetric functions and to
differential operators in $U(\mathcal W_{1+\infty})$.  Using the
Goulden--Jackson product, we extend the construction to Jack
polynomials, recover shifted Jack eigenvalues and Pieri-type
relations, and obtain explicit candidate operators in degrees three
and four.  Their real and quaternionic specializations to zonal polynomials 
are verified by Gaussian matrix integrals and exhaustive Wick enumeration.
\end{abstract}

\maketitle


\section{Introduction}
\label{sec:introduction}

The multiplication of conjugacy classes in symmetric groups is one
of the oldest concrete problems in the representation theory of
$\Sym_n$.  Through the Frobenius characteristic map, multiplication
in the center of $\CC\Sym_n$ becomes an internal product on symmetric
functions, that we shall call the class product to avoid confusion with the 
$*$-product induced by the product of characters.
  
Character theory gives, in principle, all its structure
constants, and some of the first explicit formulas go back to
Frobenius \cite{Frobenius}.  A familiar elementary example is
\begin{equation}\label{eq:intro-transposition-square}
 C_{(2,1^{n-2})}^{,2}
 =\binom n2 C_{(1^n)}
  +3C_{(3,1^{n-3})}
  +2C_{(2,2,1^{n-4})},
\end{equation}
where $C_\lambda$ denotes the sum of the permutations of cycle type
$\lambda$; terms indexed by nonexistent partitions are understood to
be zero.  Formula \eqref{eq:intro-transposition-square} already
exhibits the two features that later stability theories had to
separate: a coefficient polynomial in $n$, and coefficients that
become independent of $n$ after a suitable normalization of classes.

Farahat and Higman \cite{FH} proved that, when conjugacy classes are
indexed by reduced cycle type, their structure constants are
integer-valued polynomials in $n$.  Their universal filtered algebra
simultaneously controls the centers of all symmetric group algebras.
A different stability phenomenon emerged from the normalized classes
introduced in the work of Kerov and Olshanski \cite{KO}.  Kerov
observed that their products should have structure constants
independent of $n$; Ivanov and Kerov proved this assertion by
constructing the algebra of partial permutations \cite{IK}.  This
algebra is closely related to the algebra of shifted symmetric
functions of Okounkov and Olshanski \cite{OO}.

In a parallel development, Katriel proposed explicit forms for the
central characters and for the operators representing multiplication
by stable families of conjugacy classes \cite{Katriel}.  Goupil,
Poulalhon, and Schaeffer proved these conjectures and obtained the
corresponding differential operators on symmetric functions
\cite{GPS}.  The first member of this family is Goulden's
cut-and-join operator for multiplication by the sum of all
transpositions \cite{Goulden}.

In our earlier paper with Lascoux \cite{LT}, these questions were
approached by vertex operators.  If
\[
 \xi_j=\sum_{i<j}(i,j),
 \qquad
 \Xi_n=(\xi_1,\ldots,\xi_n),
\]
are the Jucys--Murphy elements, a single vertex operator implements,
simultaneously for every $n$, multiplication by the power sums
$p_r(\Xi_n)$.  Its coefficients generate the basic representation of
$\mathcal W_{1+\infty}$.  The differential operators associated with
the normalized stable classes can then be recovered from these
coefficients; in particular, they belong to the image of
$U(\mathcal W_{1+\infty})$.  The same paper gave another derivation of
the Goupil--Poulalhon--Schaeffer operators from Gaussian matrix
integrals and Wick's formula.

Equivalent operator algebras have subsequently appeared in other
languages.  Mironov, Morozov, and Natanzon associated differential
operators $W(\Delta)$ with Young diagrams and obtained a commutative
algebra isomorphic to the stable algebra of conjugacy classes
\cite{MMN}.  We use the word ``equivalent'' here up to the changes of
normalization customary in the Hurwitz and symmetric-function
literatures.  Another occurrence comes from quantum integrable
systems.  Dubrovin's symplectic field theory of a disk produces the
quantum dispersionless KdV, or quantum Hopf, hierarchy
\cite{DubrovinSFT}; the representation-theoretic description of this
hierarchy identifies its Hamiltonians with multiplication operators
built from Young--Jucys--Murphy elements \cite{RuzzaYang}.  The
symplectic-field-theoretic setting itself goes back to
Eliashberg--Givental--Hofer and is surveyed in Eliashberg's ICM
lecture \cite{EGH,EliashbergICM}.  These later appearances illustrate
that the stable class operators are not tied to one construction:
they sit at the intersection of symmetric-group representation
theory, Hurwitz theory, vertex algebras, and integrable systems.

The purpose of the present paper is to show that much of this theory
becomes simpler after introducing, for a homogeneous symmetric
function $f$ of degree $m$, the stable series
\begin{equation}\label{eq:intro-stable-series}
 \widehat f=\frac{m!f}{(1-p_1)^{m+1}}.
\end{equation}
For $f=p_\mu$, these series represent the normalized stable classes.
More generally, they allow one to pass directly between stable
products, shifted symmetric functions, and normally ordered
differential operators.  Since the multiplication operators for the
power sums $p_r(\Xi_n)$ belong to
$U(\mathcal W_{1+\infty})$, the same is true of the operators attached
to all the series $\widehat f$.

We then replace the classical class product by the
Goulden--Jackson product $\times_\alpha$.  This gives a Jack
deformation of the stable-series formalism.  The operators associated
with $\widehat P'_\mu$ are diagonal on the Jack basis, with eigenvalues
given by shifted Jack polynomials.  This point of view yields the
shifted Pieri formula and a recurrence of
Alexandersson--F\'eray type.  It also leads to explicit low-degree
differential operators.  The formula for $\Delta_2(\alpha)$ is exact;
the formulas displayed for $\Delta_3(\alpha)$ and
$\Delta_4(\alpha)$ are retained as computer-assisted conjectures,
with their status stated explicitly.

Finally, we study the three classical matrix specializations.  The
complex model gives $\alpha=1$, the real model gives $\alpha=2$, and
the quaternionic model gives $\alpha=1/2$.  In the real case, Wick
contractions are perfect matchings.  In the quaternionic case they
become signed M\"obius graphs, whose powers of two and signs are
controlled by their Euler characteristic.  Exhaustive enumeration in
degrees three and four agrees term by term with the corresponding
specializations of the proposed operators.  These calculations
explain why terms invisible in the permutation model occur away from
$\alpha=1$, but they are not used to claim a proof for indeterminate
$\alpha$.

The paper is organized as follows.  Section~\ref{sec:background}
recalls the vertex-operator and $\mathcal W_{1+\infty}$ background.
Sections~\ref{sec:stable-series} and
\ref{sec:shifted-symmetric-functions} develop the classical stable
series and their relation with shifted symmetric functions.
Sections~\ref{sec:goulden-jackson}--
\ref{sec:jack-differential-operators} introduce the
Goulden--Jackson deformation, shifted Jack functions, and the
associated differential operators.  The low-degree formulas are
given in Section~\ref{sec:low-degree-jack-operators}.  The real and
quaternionic matrix models are treated in
Sections~\ref{sec:real-matrix-integrals} and
\ref{sec:quaternionic-matrix-integrals}.  The precise status of the
results and the remaining problems are summarized in
Section~\ref{sec:conclusions}.

\section{Background}\label{sec:background}

In this section, we recall the results of \cite{LT} that will be used
throughout the paper.  Our main purpose is to explain why the stable
series introduced below have natural differential-operator
realizations, and why all these operators belong to the same
$\mathcal W_{1+\infty}$ framework as the operators associated with
power sums of Jucys--Murphy elements.

\subsection{Class multiplication and Jucys--Murphy elements}

Let $\Lambda$ be the algebra of symmetric functions, endowed with the
Hall scalar product, and let $\Lambda_n$ denote its homogeneous
component of degree $n$.  The Frobenius characteristic identifies
$\Lambda_n$ with the center of the group algebra $\mathbb C\mathfrak
S_n$.  We denote by $\times$ the product transported to $\Lambda_n$
from multiplication in this center.  Thus
\begin{equation}\label{eq:internal-schur}
  s_\lambda\times s_\mu
  =\delta_{\lambda\mu}\frac{1}{f^\lambda}s_\lambda,
  \qquad \lambda,\mu\vdash n,
\end{equation}
where $f^\lambda$ is the number of standard Young tableaux of shape
$\lambda$.

The Jucys--Murphy elements are
\begin{equation}
  \xi_j=\sum_{i<j}(i,j),\qquad 1\leq j\leq n,
\end{equation}
and we write $\Xi_n=(\xi_1,\ldots,\xi_n)$.  Symmetric polynomials in
the $\xi_i$ are central.  On the irreducible representation indexed
by $\lambda$, the element $f(\Xi_n)$ acts by the scalar
$f(C(\lambda))$, where $C(\lambda)$ is the multiset of contents of
$\lambda$.

The main construction of \cite{LT} is a single operator acting on
$\Lambda$ whose restriction to $\Lambda_n$ realizes multiplication
by the generating series
\begin{equation}
  \sum_{r\geq1}p_r(\Xi_n)\frac{u^r}{r!}
  =\sum_{i=1}^n\bigl(e^{u\xi_i}-1\bigr).
\end{equation}
Put $q=e^u$ and let $\mathcal E=\sum_{r\geq1}p_rD_{p_r}$ be the Euler
operator.  In bosonic notation, this operator is
\begin{equation}\label{eq:LT-Dq}
 \mathcal D(q)
 =\frac{q}{(q-1)^2}
   \sum_{m\geq1}q^{-m}
   h_m[(q-1)X]D_{h_m[(q-1)X]}-\mathcal E.
\end{equation}
Equivalently,
\begin{equation}\label{eq:LT-zero-mode}
  \mathcal D(q)
  =\frac{V_0(q)-1}{(q-1)(1-q^{-1})}-\mathcal E,
\end{equation}
where $V_0(q)$ is the zero mode of the vertex operator
\begin{equation}\label{eq:LT-vertex}
 \begin{split}
  V(z;q)
  &=\sigma_z[(q-1)X]
    D_{\sigma_{1/z}[(1-q^{-1})X]}\\
  &=:\!\exp\left(
       \sum_{r\neq0}\frac{1-q^{-r}}{r}z^{-r}\alpha_r
     \right)\!:\ .
 \end{split}
\end{equation}
Here $\alpha_{-r}=p_r$, $\alpha_r=D_{p_r}=r\,\partial/\partial p_r$
for $r>0$.

Writing
\begin{equation}\label{eq:D-expansion}
  \mathcal D(e^u)=\sum_{r\geq1}\mathcal D_r\frac{u^r}{r!},
\end{equation}
the restriction of $\mathcal D_r$ to $\Lambda_n$ is multiplication by
$p_r(\Xi_n)$.  Thus the infinitely many class-multiplication
operators, for all symmetric groups simultaneously, are encoded by
the coefficients of one vertex operator.

\subsection{The $\mathcal W_{1+\infty}$ realization}

Under the boson--fermion correspondence, the vertex operator
\eqref{eq:LT-vertex} belongs to the basic representation ${\widehat r}_0$ of the
universal central extension $\widehat{\mathcal D}$ of the Lie algebra
of differential operators on the circle.  We use the standard
notation
\begin{equation}
  \widehat{\mathcal D}=\mathcal W_{1+\infty}.
\end{equation}
More explicitly, if
\begin{equation}
  W_k(q)=\widehat r_0\left(\sum_{i\in\mathbb Z}q^iE_{i,i+k}\right),
\end{equation}
then the modes of $V(z;q)$ are proportional to the $W_k(q)$, and
their expansion at $q=1$ yields the standard generators of
$\mathcal W_{1+\infty}$.  In particular, every coefficient
$\mathcal D_r$ in \eqref{eq:D-expansion} belongs to the image of
$U(\mathcal W_{1+\infty})$ in the charge-zero bosonic Fock space.

It is useful to introduce the commuting diagonal operators
\begin{equation}\label{eq:shifted-P}
  P_r=\widehat r_0\left(\sum_{i\in\mathbb Z}i^rE_{ii}\right).
\end{equation}
Their eigenvalues on $s_\lambda$ are the shifted power sums
\begin{equation}\label{eq:shifted-power-sums}
  \widetilde p_r(\lambda)
  =\sum_{i\geq1}\bigl((\lambda_i-i)^r-(-i)^r\bigr).
\end{equation}
The two families of operators are related by
\begin{equation}\label{eq:P-D-relation}
  P_r=\sum_{j=0}^{r-1}\binom{r}{j}\mathcal D_j,
  \qquad \mathcal D_0=\mathcal E.
\end{equation}
Consequently, the algebra of polynomials in the $P_r$ is a
commutative subalgebra of the image of
$U(\mathcal W_{1+\infty})$.

\subsection{Stable series and their operators}

Let $f\in\Lambda_m$ be homogeneous of degree $m$.  We associate with
$f$ the stable symmetric series
\begin{equation}\label{eq:hat-f}
  \widehat f
  =\frac{m!\,f}{(1-p_1)^{m+1}}
  =\sum_{n\geq m}(n)_m f\,p_1^{n-m},
\end{equation}
where $(n)_m=n(n-1)\cdots(n-m+1)$.  In particular,
\begin{equation}
  \widehat p_\rho
  =\sum_{n\geq |\rho|}(n)_{|\rho|}
     p_{\rho,1^{n-|\rho|}}.
\end{equation}
When $\rho$ has no parts equal to $1$, the degree-$n$ component of
$\widehat p_\rho$ is the Frobenius characteristic of the normalized
conjugacy class of cycle type $\rho1^{n-|\rho|}$.

For every partition $\rho$, let $\mathcal A_\rho$ be the operator
whose restriction to $\Lambda_n$ is $\times$-multiplication by the
degree-$n$ component of $\widehat p_\rho$.  These operators are
simultaneously diagonal in the Schur basis:
\begin{equation}\label{eq:A-rho-eigenvalue}
  \mathcal A_\rho s_\lambda
  =f_\rho(\lambda)s_\lambda,
\end{equation}
where $f_\rho$ is the normalized central character
\begin{equation}
  f_\rho(\lambda)
  =(n)_{|\rho|}
    \frac{\chi^\lambda_{\rho,1^{n-|\rho|}}}{f^\lambda}.
\end{equation}

Kerov and Olshanski proved that $f_\rho$ is a shifted symmetric
function.  Hence it is a polynomial in the shifted power sums
\eqref{eq:shifted-power-sums}.  Since the $P_r$ have precisely these
eigenvalues, there exists a polynomial $F_\rho$ such that
\begin{equation}\label{eq:A-polynomial-P}
  \mathcal A_\rho=F_\rho(P_1,P_2,\ldots).
\end{equation}
Equations \eqref{eq:P-D-relation} and \eqref{eq:A-polynomial-P}
imply the fact that will be essential below:
\begin{equation}\label{eq:A-in-W}
  {
  \mathcal A_\rho\in
  \widehat{r}_0 \left(U(\mathcal W_{1+\infty})\right)
  \quad\text{for every partition }\rho,}
\end{equation}
More generally, if
\begin{equation}
  f=\sum_{\rho\vdash m}c_\rho p_\rho,
\end{equation}
then
\begin{equation}
  \widehat f=\sum_{\rho\vdash m}c_\rho\widehat p_\rho
\end{equation}
and the associated multiplication operator is
\begin{equation}
  \mathcal A_f=\sum_{\rho\vdash m}c_\rho\mathcal A_\rho.
\end{equation}
It follows at once from \eqref{eq:A-in-W} that
\begin{equation}\label{eq:Af-in-W}
  {
  \mathcal A_f\in
  \widehat{r}_0 \left(U(\mathcal W_{1+\infty})\right)
  \quad\text{for every homogeneous }f\in\Lambda.}
\end{equation}
Thus the differential operators attached to the stable series
$\widehat f$ do not form a new family external to the vertex-operator
formalism of \cite{LT}: they all lie in the same
$\mathcal W_{1+\infty}$ representation as the operators associated
with power sums of Jucys--Murphy elements.

The purpose of the next sections is to reverse the usual logic.
Rather than deriving stability from the theory of shifted symmetric
functions, we shall start from the explicit differential operators
and the elementary form \eqref{eq:hat-f} of the stable series.  This
will recover directly the Ivanov--Kerov algebra and, subsequently,
the shifted-symmetric-function formalism.


\section{Stable symmetric series and normalized classes}
\label{sec:stable-series}

We now use the differential operators recalled in
\cref{sec:background} to give a direct realization of the stable
algebra of normalized conjugacy classes.  The argument is elementary:
the action of every normally ordered monomial differential operator
on the rational series $\widehat p_\mu$ can be computed explicitly,
and the result is again an integral linear combination of series of
the same form.

\subsection{The completed class product}

We work in the degree completion
\begin{equation}
  \widehat\Lambda=\prod_{n\geq0}\Lambda_n.
\end{equation}
The class product $\times$ extends componentwise to
$\widehat\Lambda$.  Thus, if $F=\sum_nF_n$ and $G=\sum_nG_n$, with
$F_n,G_n\in\Lambda_n$, then
\begin{equation}
  F\times G=\sum_{n\geq0}F_n\times G_n.
\end{equation}

For a partition $\mu$ of size $m$, set
\begin{equation}\label{eq:stable-p-mu}
  \widehat p_\mu
  =\frac{m!p_\mu}{(1-p_1)^{m+1}}
  =\sum_{n\geq m}(n)_m p_{\mu,1^{n-m}}.
\end{equation}
Notice that the parts equal to $1$ occurring in $\mu$ are retained:
they encode distinguished fixed points in the corresponding partial
permutation.  Let
\begin{equation}
  \mathscr A=\operatorname{span}_{\mathbb C}
  \{\widehat p_\mu\mid \mu\text{ a partition}\}
  \subset\widehat\Lambda.
\end{equation}
The series $\widehat p_\mu$ are linearly independent, since their
lowest homogeneous components are $m!p_\mu$.

For later reference, a dual family can also be written explicitly.
If $\rho$ has no parts equal to $1$, $|\rho|=r$, and $k\geq0$, put
\begin{equation}\label{eq:dual-stable-p}
  b_{\rho1^k}
  =\sum_{i=0}^k
    \frac{(-1)^{k-i}}{(k-i)!}
    \frac{p^*_{\rho1^i}}{(r+i)!},
\end{equation}
where $p^*_\lambda=p_\lambda/z_\lambda$ is dual to $p_\lambda$ for
the Hall scalar product.  A direct coefficient extraction gives
\begin{equation}
  \scalar{\widehat p_\alpha}{b_\beta}
  =\delta_{\alpha\beta}.
\end{equation}

\subsection{Differential operators for normalized classes}

Let $\mu\vdash m$, and choose a permutation $\sigma\in\Sym_m$ of
cycle type $\mu$.  For $L=(l_1,\ldots,l_m)\in(\mathbb Z_{>0})^m$
and $\tau\in\Sym_m$, define
\begin{equation}\label{eq:p-L-tau}
  p_L^\tau
  =\prod_{c\in\operatorname{Cycles}(\tau)}
    p_{\sum_{i\in c}l_i}.
\end{equation}
The operator implementing $\times$-multiplication by
$\widehat p_\mu$ is
\begin{equation}\label{eq:A-mu-explicit}
  \mathcal A_\mu
  =\sum_{l_1,\ldots,l_m\geq1}
    \sum_{\tau\in\Sym_m}
    p_L^{\sigma\tau}D_{p_L^\tau}.
\end{equation}
Although the sum is infinite, its action on each homogeneous
component is finite.  Formula \eqref{eq:A-mu-explicit} is the
differential form of the Wick expansion obtained in \cite{LT}; for
reduced cycle types it is equivalent, up to the usual factor
$z_\mu$, to the operators of Goupil--Poulalhon--Schaeffer
\cite{GPS}.

Every summand of \eqref{eq:A-mu-explicit} is a degree-preserving
normally ordered monomial $p_\alpha D_{p_\beta}$, with a
nonnegative integral coefficient.  The following observation is the
basic stability mechanism.

\begin{lemma}\label{lem:monomial-stability}
Let $\rho$ be a partition without parts equal to $1$, let
$r=|\rho|$, and let $l,k\geq0$.  Suppose that
$\beta=\bar\beta1^k$, where $\bar\beta$ has no parts equal to $1$,
and that $p_\alpha D_{p_\beta}$ preserves the usual degree.  Then
\begin{equation}
 p_\alpha D_{p_\beta}\widehat p_{\rho1^l}
\end{equation}
is a nonnegative integral linear combination of stable series
$\widehat p_\gamma$.
\end{lemma}

\begin{proof}
Put $N=r+l$.  Since $D_{p_1}=\partial/\partial p_1$, the Leibniz
rule gives
\begin{align}
&p_\alpha D_{p_{\bar\beta}p_1^k}
 \frac{N!p_\rho p_1^l}{(1-p_1)^{N+1}}
\notag\\
&\quad =p_\alpha(D_{p_{\bar\beta}}p_\rho)N!
 \sum_{i+j=k}\binom{k}{j}(l)_j p_1^{l-j}
 \frac{(N+i)_i}{(1-p_1)^{N+i+1}}.
\label{eq:monomial-action}
\end{align}
Here $(N+i)_i=(N+i)!/N!$.  Consequently the factorials in
\eqref{eq:monomial-action} combine to give
\begin{equation}
 \sum_{i+j=k}\binom{k}{j}(l)_j
 \frac{(N+i)!,
 p_\alpha(D_{p_{\bar\beta}}p_\rho)p_1^{l-j}}
 {(1-p_1)^{N+i+1}}.
\label{eq:monomial-action-normalized}
\end{equation}
Because the operator preserves degree, every power-sum monomial in
the numerator of the summand indexed by $i+j=k$ has degree $N+i$.
Moreover, $D_{p_{\bar\beta}}p_\rho$ is a nonnegative integral linear
combination of power-sum monomials.  Each term in
\eqref{eq:monomial-action-normalized} is therefore a nonnegative
integral multiple of some $\widehat p_\gamma$.
\end{proof}

\subsection{The stable algebra}

\begin{theorem}\label{thm:stable-algebra}
The vector space $\mathscr A$ is a commutative subalgebra of
$(\widehat\Lambda,\times)$.  More precisely, there exist
nonnegative integers $d_{\mu\nu}^{\lambda}$, independent of the
ambient degree, such that
\begin{equation}\label{eq:stable-structure-constants}
  \widehat p_\mu\times\widehat p_\nu
  =\sum_\lambda d_{\mu\nu}^{\lambda}\widehat p_\lambda.
\end{equation}
Only finitely many terms occur, and
\begin{equation}\label{eq:stable-filtration}
  \max(|\mu|,|\nu|)\leq|\lambda|\leq|\mu|+|\nu|
\end{equation}
whenever $d_{\mu\nu}^{\lambda}\neq0$.
\end{theorem}

\begin{proof}
By definition,
\begin{equation}
  \widehat p_\mu\times\widehat p_\nu
  =\mathcal A_\mu(\widehat p_\nu).
\end{equation}
Expand $\mathcal A_\mu$ by
\eqref{eq:A-mu-explicit}.  Each summand satisfies the hypotheses of
\cref{lem:monomial-stability}; hence the result belongs to
$\mathscr A$ and has nonnegative integral coefficients in the basis
$\{\widehat p_\lambda\}$.  The number of contributing summands is
finite because a derivative $D_{p_L^\tau}$ can act nontrivially only
when its non-$p_1$ factors occur in the numerator of
$\widehat p_\nu$.  The degree bounds follow directly from
\eqref{eq:monomial-action-normalized}: contractions can increase the
stable degree by at most $|\mu|$, and symmetry in $\mu,\nu$ gives the
lower bound.
\end{proof}

Taking the homogeneous component of degree $n$ in
\eqref{eq:stable-structure-constants} gives the same structure
constants for the normalized conjugacy classes of every symmetric
group $\Sym_n$.  Thus \cref{thm:stable-algebra} recovers the stable
class algebra of Ivanov and Kerov \cite{IK}.  Their partial
permutations give a direct combinatorial interpretation of the
integers $d_{\mu\nu}^{\lambda}$; the argument above shows that their
existence, stability, and nonnegativity are already visible in the
differential operators.

\begin{example}\label{ex:p2-times-p2}
For $\mu=(2)$, formula \eqref{eq:A-mu-explicit} gives twice
Goulden's cut-and-join operator:
\begin{equation}
  \mathcal A_2
  =\sum_{i,j\geq1}
  \left(p_ip_jD_{p_{i+j}}+p_{i+j}D_{p_i}D_{p_j}\right).
\end{equation}
Only three terms can act nontrivially on
$\widehat p_2=2p_2/(1-p_1)^3$, namely
\begin{equation}
  p_1^2D_{p_2},\qquad
  2p_3D_{p_2}D_{p_1},\qquad
  p_2D_{p_1}^2.
\end{equation}
They give
\begin{align}
 \widehat p_2\times\widehat p_2
 &=\mathcal A_2\widehat p_2\notag\\
 &=\frac{4p_1^2}{(1-p_1)^3}
   +\frac{24p_3}{(1-p_1)^4}
   +\frac{24p_2^2}{(1-p_1)^5}\notag\\
 &=2\widehat p_{11}+4\widehat p_3+\widehat p_{22}.
\label{eq:p2-times-p2}
\end{align}
The three terms correspond respectively to retaining two marked fixed
points, joining a transposition to a marked fixed point, and retaining
two transpositions.
\end{example}

The stable algebra will next be identified with the algebra of
shifted symmetric functions.  From the present point of view, shifted
functions arise as the joint eigenvalues of the operators
$\mathcal A_\mu$, rather than as an external device used to prove
stability.

\section{Stable Jack series and shifted Jack functions}
\label{sec:shifted-jack-functions}

The orthogonal idempotents of \cref{sec:goulden-jackson} allow the
classical construction of \cref{sec:shifted-symmetric-functions} to
be repeated for Jack polynomials.  The only additional difficulty is
normalization.  We first pass from the idempotent normalization $Q'$
to the normalization naturally adapted to shifted Jack polynomials.

\subsection{The $P'$ normalization and stable lifts}

For a Jack polynomial $P_\lambda^{(\alpha)}$, define
\begin{equation}\label{eq:Pprime-definition}
  P'_\lambda(X)
  =P_\lambda^{(\alpha)}[\alpha X],
\end{equation}
where $\alpha$ is a plethystic scalar, so that
$p_r[\alpha X]=\alpha p_r(X)$.  In terms of the Hall-dual basis of
\cref{eq:Qprime-duality}, one has
\begin{equation}\label{eq:Pprime-Qprime}
  P'_\lambda
  =\frac{c'_\lambda(\alpha)}{c_\lambda(\alpha)}Q'_\lambda.
\end{equation}
Indeed, the corresponding integral form satisfies
\begin{equation}
  J_\lambda^{(\alpha)}[\alpha X]
  =c'_\lambda(\alpha)Q'_\lambda(X).
\end{equation}
In particular,
\begin{equation}\label{eq:Pprime-one}
  P'_{(1)}=\alpha p_1.
\end{equation}

If $\mu\vdash m$, define its stable Jack lift by
\begin{equation}\label{eq:stable-Pprime}
  \widehat P'_\mu
  =\frac{m!P'_\mu}{(1-p_1)^{m+1}}
  =\sum_{n\geq m}(n)_mP'_\mu p_1^{n-m}.
\end{equation}
Since $P'_\mu$ is proportional to $Q'_\mu$, multiplication by
$\widehat P'_\mu$ is diagonal in the $Q'$-basis, or equivalently in
the $P'$-basis.

\subsection{Shifted Jack eigenvalues}

We use the notation $P_\mu^\#(x;\alpha)$ of Alexandersson and
F\'eray for the shifted Jack polynomial with Jack parameter $\alpha$.
In the notation of Okounkov and Olshanski,
\begin{equation}\label{eq:Psharp-Pstar}
  P_\mu^\#(x;\alpha)=P_\mu^*(x;1/\alpha).
\end{equation}
Thus the change from a star to a sharp records the change from their
parameter $\theta$ to the standard Jack parameter $\alpha=1/\theta$.
The polynomial $P_\mu^\#$ is characterized by the interpolation
conditions
\begin{equation}\label{eq:shifted-Jack-vanishing}
  P_\mu^\#(\lambda;\alpha)=0
  \qquad\text{unless }\mu\subseteq\lambda,
\end{equation}
together with highest homogeneous component $P_\mu^{(\alpha)}$ and
the normalization
\begin{equation}\label{eq:Psharp-diagonal}
  P_\mu^\#(\mu;\alpha)
  =\alpha^{-|\mu|}c'_\mu(\alpha).
\end{equation}
The corresponding integral normalization is
$J_\mu^\#=c_\mu(\alpha)P_\mu^\#$.  Equivalently, the binomial
characterization used in \cite{OOJack,AF} gives, for $\lambda\vdash n$
and $m=|\mu|$,
\begin{equation}\label{eq:shifted-Jack-coefficient-formula}
 \alpha^mP_\mu^\#(\lambda;\alpha)
 =\frac{(n)_m c_\lambda(\alpha)}{n!}
  \scalar{P_\lambda^{(\alpha)}}
  {P'_\mu p_1^{n-m}}_1.
\end{equation}

\begin{theorem}\label{thm:stable-Pprime-spectrum}
Let $\mu\vdash m$ and $\lambda\vdash n$.  Then
\begin{equation}\label{eq:stable-Pprime-spectrum}
  \widehat P'_\mu\times_\alpha P'_\lambda
  =\alpha^mP_\mu^\#(\lambda;\alpha)P'_\lambda.
\end{equation}
Consequently, the stable span of the $\widehat P'_\mu$ is a
commutative algebra for $\times_\alpha$, and the correspondence
\begin{equation}\label{eq:stable-shifted-Jack-map}
  \widehat P'_\mu
  \longmapsto
  \alpha^{|\mu|}P_\mu^\#
\end{equation}
identifies it with the algebra of shifted Jack functions.
\end{theorem}

\begin{proof}
The degree-$n$ component of $\widehat P'_\mu$ is
$(n)_mP'_\mu p_1^{n-m}$.  Its coefficient on $Q'_\lambda$ is, by
Hall duality,
\begin{equation}
 (n)_m
 \scalar{P_\lambda^{(\alpha)}}{P'_\mu p_1^{n-m}}_1.
\end{equation}
Formula \eqref{eq:GJ-general-eigenvalue} therefore shows that its
eigenvalue on $P'_\lambda$ is
\begin{equation}
 \frac{(n)_m c_\lambda(\alpha)}{n!}
 \scalar{P_\lambda^{(\alpha)}}{P'_\mu p_1^{n-m}}_1.
\end{equation}
By \eqref{eq:shifted-Jack-coefficient-formula}, this is
$\alpha^mP_\mu^\#(\lambda;\alpha)$, proving
\eqref{eq:stable-Pprime-spectrum}.  The last assertion follows by
multiplying eigenvalues, exactly as in
\cref{prop:stable-shifted-isomorphism}.
\end{proof}

Thus shifted Jack functions arise here for the same reason as shifted
Schur functions: they are the eigenvalues of stable multiplication
operators.  Notice that the factor $\alpha^m$ in
\eqref{eq:stable-Pprime-spectrum} is forced by the plethystic
normalization \eqref{eq:Pprime-definition}.

\subsection{The Jack Pieri rule}

Write the ordinary Jack Pieri rule in the $P'$ normalization as
\begin{equation}\label{eq:Pprime-Pieri}
  P'_\mu P'_{(1)}
  =\sum_{\nu/\mu=\square}
    \psi'_{\nu/\mu}(\alpha)P'_\nu.
\end{equation}
In view of \eqref{eq:Pprime-one}, this is equivalent to
\begin{equation}
  \alpha P'_\mu p_1
  =\sum_{\nu/\mu=\square}
    \psi'_{\nu/\mu}(\alpha)P'_\nu.
\end{equation}

\begin{theorem}[Stable Jack Pieri rule]
\label{thm:stable-Jack-Pieri}
For $\mu\vdash m$,
\begin{equation}\label{eq:stable-Jack-Pieri}
 \widehat P'_\mu\times_\alpha\widehat P'_{(1)}
 =\alpha m\widehat P'_\mu
  +\sum_{\nu/\mu=\square}
    \psi'_{\nu/\mu}(\alpha)\widehat P'_\nu.
\end{equation}
Equivalently, the shifted Jack polynomials satisfy
\begin{equation}\label{eq:shifted-Jack-Pieri}
 P_\mu^\#P_{(1)}^\#
 =mP_\mu^\#
  +\sum_{\nu/\mu=\square}
    \psi'_{\nu/\mu}(\alpha)P_\nu^\#.
\end{equation}
\end{theorem}

\begin{proof}
The degree-$N$ component of $\widehat P'_{(1)}$ is
$\alpha Np_1^N$.  Since $p_1^N$ is the identity for
$\times_\alpha$ on $\Lambda_N$, one obtains
\begin{align}
\widehat P'_\mu\times_\alpha\widehat P'_{(1)}
 &=\alpha\sum_{k\geq0}(m+k)
   \frac{(m+k)!}{k!}P'_\mu p_1^k\notag\\
 &=\alpha m\widehat P'_\mu
   +\alpha\sum_{k\geq1}
    \frac{(m+k)!}{(k-1)!}
    (P'_\mu p_1)p_1^{k-1}.
\end{align}
Apply \eqref{eq:Pprime-Pieri} to the second term to obtain
\eqref{eq:stable-Jack-Pieri}.  Finally, apply both sides to
$P'_\lambda$ and use \eqref{eq:stable-Pprime-spectrum}; after
dividing by $\alpha^{m+1}$, one obtains
\eqref{eq:shifted-Jack-Pieri}.
\end{proof}

\subsection{The Alexandersson--F\'eray recurrence}

Define the shifted Jack structure constants by
\begin{equation}\label{eq:shifted-Jack-structure-constants}
  P_\mu^\#P_\nu^\#
  =\sum_\lambda
    c_{\mu\nu}^{\lambda}(\alpha)P_\lambda^\#.
\end{equation}
The same associativity argument as in
\cref{thm:molev-sagan-recurrence} gives the Jack deformation of the
Molev--Sagan recurrence.

\begin{theorem}[Alexandersson--F\'eray]
\label{thm:Alexandersson-Feray}
Let $\nu\vdash n$ and $\lambda\vdash l$, with $l>n$.  Then
\begin{align}
 (l-n)c_{\mu\nu}^{\lambda}(\alpha)
 &={}
 \sum_{\nu^+/\nu=\square}
  \psi'_{\nu^+/\nu}(\alpha)
  c_{\mu\nu^+}^{\lambda}(\alpha)
 \notag\\
 &\quad-
 \sum_{\lambda/\lambda^-=\square}
  \psi'_{\lambda/\lambda^-}(\alpha)
  c_{\mu\nu}^{\lambda^-}(\alpha).
\label{eq:Alexandersson-Feray}
\end{align}
The initial conditions are
\begin{equation}\label{eq:AF-initial}
 c_{\mu\nu}^{\lambda}(\alpha)
 =\delta_{\lambda\nu}P_\mu^\#(\nu;\alpha)
 \qquad (|\lambda|=|\nu|).
\end{equation}
The recurrence and the initial conditions determine all the
coefficients.
\end{theorem}

\begin{proof}
Multiply \eqref{eq:shifted-Jack-structure-constants} by
$P_{(1)}^\#$ and use \eqref{eq:shifted-Jack-Pieri}.  Expanding first on
the $\nu$-factor, the coefficient of $P_\lambda^\#$ is
\begin{equation}
 n c_{\mu\nu}^{\lambda}
 +\sum_{\nu^+/\nu=\square}
   \psi'_{\nu^+/\nu}c_{\mu\nu^+}^{\lambda}.
\end{equation}
Expanding instead after the product
$P_\mu^\#P_\nu^\#$, the same coefficient is
\begin{equation}
 l c_{\mu\nu}^{\lambda}
 +\sum_{\lambda/\lambda^-=\square}
   \psi'_{\lambda/\lambda^-}c_{\mu\nu}^{\lambda^-}.
\end{equation}
Equating the two expressions gives
\eqref{eq:Alexandersson-Feray}.  The initial values follow from the
vanishing property \eqref{eq:shifted-Jack-vanishing}.  As before, if
$d=|\lambda|-|\nu|>0$, every coefficient on the right-hand side has
difference $d-1$, so induction on $d$ proves uniqueness.
\end{proof}

At $\alpha=1$, $P'_\lambda=s_\lambda$,
$\psi'_{\nu/\mu}(1)=1$, and the results of this section reduce to the
shifted Schur formulas of \cref{sec:shifted-symmetric-functions}.
The next step is to express specific stable Jack series, notably the
power-sum series $\widehat p_\mu$, by differential operators and to
compare their eigenvalues with Jack characters.


\section{The Goulden--Jackson deformation}
\label{sec:goulden-jackson}

We now introduce the deformation of class multiplication defined by
Goulden and Jackson \cite{GJ96}.  Two scalar products occur in this
construction and must be kept separate.  Jack polynomials are
orthogonal for the Jack scalar product, whereas the product
$\times_\alpha$ is dual to its coproduct for the ordinary Hall scalar
product.  We fix these conventions explicitly before introducing
stable Jack series.

\subsection{Jack scalar products and normalizations}

Let $\alpha$ be an indeterminate.  The Jack scalar product is defined
on power sums by
\begin{equation}\label{eq:jack-scalar-product}
 \scalar{p_\lambda}{p_\mu}_\alpha
 =\delta_{\lambda\mu}\,z_\lambda
  \alpha^{\ell(\lambda)}.
\end{equation}
We write $P_\lambda^{(\alpha)}$ for the monic Jack polynomial and
\begin{equation}
 J_\lambda^{(\alpha)}=c_\lambda(\alpha)P_\lambda^{(\alpha)}
\end{equation}
for the integral form, where
\begin{equation}\label{eq:jack-hook-products}
 \begin{split}
 c_\lambda(\alpha)
 &=\prod_{\square\in\lambda}
   \bigl(\alpha a(\square)+l(\square)+1\bigr),\\
 c'_\lambda(\alpha)
 &=\prod_{\square\in\lambda}
   \bigl(\alpha a(\square)+l(\square)+\alpha\bigr).
 \end{split}
\end{equation}
Thus
\begin{equation}\label{eq:jack-norm}
 j_\lambda(\alpha)
 =\scalar{J_\lambda^{(\alpha)}}{J_\lambda^{(\alpha)}}_\alpha
 =c_\lambda(\alpha)c'_\lambda(\alpha).
\end{equation}

In addition to \eqref{eq:jack-scalar-product}, we use the ordinary
Hall scalar product
\begin{equation}\label{eq:hall-scalar-product-again}
 \scalar{p_\lambda}{p_\mu}_1
 =\delta_{\lambda\mu}z_\lambda.
\end{equation}
Let $Q'_\lambda=Q_\lambda^{\prime(\alpha)}$ be the basis Hall-dual to
the Jack $P$-basis:
\begin{equation}\label{eq:Qprime-duality}
 \scalar{P_\lambda^{(\alpha)}}{Q'_\mu}_1
 =\delta_{\lambda\mu}.
\end{equation}
This is the $Q'$ normalization used below.  It is also the basis
implemented as \texttt{JackQp} in Sage, with the parameter
\texttt{t} equal to our $\alpha$.  For example,
\begin{equation}
 Q'_{(1,1)}
 =\frac{\alpha}{\alpha+1}p_1^2
  -\frac{1}{\alpha+1}p_2.
\end{equation}

\subsection{The Goulden--Jackson kernel}

Fix a degree $n$.  Goulden and Jackson define their deformed
connection coefficients through the three-alphabet series
\begin{equation}\label{eq:GJ-three-alphabet-kernel}
 \Phi_{\alpha,n}(X,Y,Z)
 =\sum_{\lambda\vdash n}
   \frac{J_\lambda^{(\alpha)}(X)
         J_\lambda^{(\alpha)}(Y)
         J_\lambda^{(\alpha)}(Z)}
        {j_\lambda(\alpha)}.
\end{equation}
They are the coefficients $a_{\mu\nu}^{\rho}(\alpha)$ determined by
\begin{equation}\label{eq:GJ-connection-coefficients}
 \Phi_{\alpha,n}(X,Y,Z)
 =\sum_{\rho,\mu,\nu\vdash n}
   a_{\mu\nu}^{\rho}(\alpha)
   \frac{p_\rho(X)p_\mu(Y)p_\nu(Z)}
        {z_\rho\alpha^{\ell(\rho)}}.
\end{equation}
At $\alpha=1$, the Jack polynomials reduce to Schur functions up to
the hook normalization, and
$a_{\mu\nu}^{\rho}(1)$ are the ordinary connection coefficients of
the center of $\mathbb C\Sym_n$.

For our purposes, it is more convenient to encode the same
construction by a coproduct.  Define
\begin{equation}\label{eq:Gamma-alpha-J}
 \Gamma_\alpha\bigl(J_\lambda^{(\alpha)}\bigr)
 =\frac{1}{n!}
  J_\lambda^{(\alpha)}\otimes J_\lambda^{(\alpha)},
 \qquad \lambda\vdash n.
\end{equation}
The Goulden--Jackson product $\times_\alpha$ on $\Lambda_n$ is the
product dual to $\Gamma_\alpha$ for the \emph{ordinary} Hall scalar
product:
\begin{equation}\label{eq:GJ-duality}
 \scalar{f\times_\alpha g}{h}_1
 =\scalar{f\otimes g}{\Gamma_\alpha(h)}_{1\otimes1}.
\end{equation}
This convention is responsible for some normalization factors, but
has the advantage that at $\alpha=1$ it is exactly the class product
$\times$ used in the preceding sections.

\begin{proposition}\label{prop:GJ-associative}
For every $n$, the product $\times_\alpha$ is commutative and
associative.
\end{proposition}

\begin{proof}
Formula \eqref{eq:Gamma-alpha-J} shows that $\Gamma_\alpha$ is
cocommutative and coassociative on a basis of $\Lambda_n$.  Its Hall
dual is therefore commutative and associative.
\end{proof}

\subsection{Orthogonal idempotents}

The Hall-dual basis $Q'_\lambda$ diagonalizes the
Goulden--Jackson product.

\begin{theorem}\label{thm:Qprime-idempotents}
For $\lambda,\mu\vdash n$,
\begin{equation}\label{eq:Qprime-product}
 n!Q'_\lambda\times_\alpha n!Q'_\mu
 =\delta_{\lambda\mu},
  n!c_\lambda(\alpha)Q'_\lambda.
\end{equation}
Consequently,
\begin{equation}\label{eq:GJ-idempotents}
 E_\lambda^{(\alpha)}
 =\frac{n!}{c_\lambda(\alpha)}Q'_\lambda
\end{equation}
are pairwise orthogonal idempotents.
\end{theorem}

\begin{proof}
Since $J_\rho^{(\alpha)}=c_\rho(\alpha)P_\rho^{(\alpha)}$,
\eqref{eq:Qprime-duality} gives
\begin{equation}
 \scalar{Q'_\lambda}{J_\rho^{(\alpha)}}_1
 =c_\rho(\alpha)\delta_{\lambda\rho}.
\end{equation}
Using \eqref{eq:GJ-duality} and
\eqref{eq:Gamma-alpha-J}, we obtain
\begin{align}
&\scalar{Q'_\lambda\times_\alpha Q'_\mu}
        {J_\rho^{(\alpha)}}_1\notag\\
&\qquad =\frac{1}{n!}
 \scalar{Q'_\lambda}{J_\rho^{(\alpha)}}_1
 \scalar{Q'_\mu}{J_\rho^{(\alpha)}}_1\notag\\
&\qquad =\frac{c_\rho(\alpha)^2}{n!}
 \delta_{\lambda\rho}\delta_{\mu\rho}.
\end{align}
As the $Q'_\rho$ are Hall-dual to the $P_\rho^{(\alpha)}$, this is
equivalent to
\begin{equation}
 Q'_\lambda\times_\alpha Q'_\mu
 =\delta_{\lambda\mu}
  \frac{c_\lambda(\alpha)}{n!}Q'_\lambda,
\end{equation}
which is \eqref{eq:Qprime-product}.  Formula
\eqref{eq:GJ-idempotents} follows immediately.
\end{proof}

It follows that every multiplication operator for
$\times_\alpha$ is diagonal in the $Q'$-basis.  More explicitly, if
\begin{equation}
 F=\sum_{\lambda\vdash n}u_\lambda Q'_\lambda,
\end{equation}
then
\begin{equation}\label{eq:GJ-general-eigenvalue}
 F\times_\alpha Q'_\mu
 =u_\mu\frac{c_\mu(\alpha)}{n!}Q'_\mu.
\end{equation}
Thus the problem of determining the spectrum of a stable series is
reduced to computing its expansion in the $Q'$-basis.

\begin{remark}\label{rem:GJ-alpha-one}
At $\alpha=1$, one has $P_\lambda^{(1)}=Q'_\lambda=s_\lambda$ and
$c_\lambda(1)=H_\lambda=n!/f^\lambda$.  Formula
\eqref{eq:Qprime-product} becomes
\begin{equation}
 s_\lambda\times s_\mu
 =\delta_{\lambda\mu}\frac{1}{f^\lambda}s_\lambda,
\end{equation}
as in \eqref{eq:internal-schur}.
\end{remark}

\subsection{Connection coefficients and the $b$-parameter}

It is customary to set
\begin{equation}
 b=\alpha-1.
\end{equation}
Goulden and Jackson conjectured that, after the standard
normalization, the coefficients
$a_{\mu\nu}^{\rho}(1+b)$ are polynomials in $b$ with nonnegative
integer coefficients, and that they enumerate matchings weighted by
a statistic of non-orientability.  We shall return to this conjecture
after deriving the low-degree differential operators and studying the
specializations $\alpha=2$ and $\alpha=1/2$.  

We next extend the stable-series formalism to
$\times_\alpha$.  The orthogonal idempotents
\eqref{eq:GJ-idempotents} will identify the eigenvalues of stable Jack
series with shifted Jack polynomials, exactly as the Schur
idempotents did in \cref{sec:shifted-symmetric-functions}.

\section{Stable Jack series and shifted Jack functions}
\label{sec:shifted-jack-functions}

The orthogonal idempotents of \cref{sec:goulden-jackson} allow the
classical construction of \cref{sec:shifted-symmetric-functions} to
be repeated for Jack polynomials.  The only additional difficulty is
normalization.  We first pass from the idempotent normalization $Q'$
to the normalization naturally adapted to shifted Jack polynomials.

\subsection{The $P'$ normalization and stable lifts}

For a Jack polynomial $P_\lambda^{(\alpha)}$, define
\begin{equation}\label{eq:Pprime-definition}
  P'_\lambda(X)
  =P_\lambda^{(\alpha)}[\alpha X],
\end{equation}
where $\alpha$ is a plethystic scalar, so that
$p_r[\alpha X]=\alpha p_r(X)$.  In terms of the Hall-dual basis of
\cref{eq:Qprime-duality}, one has
\begin{equation}\label{eq:Pprime-Qprime}
  P'_\lambda
  =\frac{c'_\lambda(\alpha)}{c_\lambda(\alpha)}Q'_\lambda.
\end{equation}
Indeed, the corresponding integral form satisfies
\begin{equation}
  J_\lambda^{(\alpha)}[\alpha X]
  =c'_\lambda(\alpha)Q'_\lambda(X).
\end{equation}
In particular,
\begin{equation}\label{eq:Pprime-one}
  P'_{(1)}=\alpha p_1.
\end{equation}

If $\mu\vdash m$, define its stable Jack lift by
\begin{equation}\label{eq:stable-Pprime}
  \widehat P'_\mu
  =\frac{m!P'_\mu}{(1-p_1)^{m+1}}
  =\sum_{n\geq m}(n)_mP'_\mu p_1^{n-m}.
\end{equation}
Since $P'_\mu$ is proportional to $Q'_\mu$, multiplication by
$\widehat P'_\mu$ is diagonal in the $Q'$-basis, or equivalently in
the $P'$-basis.

\subsection{Shifted Jack eigenvalues}

We use the notation $P_\mu^\#(x;\alpha)$ of Alexandersson and
F\'eray for the shifted Jack polynomial with Jack parameter $\alpha$.
In the notation of Okounkov and Olshanski,
\begin{equation}\label{eq:Psharp-Pstar}
  P_\mu^\#(x;\alpha)=P_\mu^*(x;1/\alpha).
\end{equation}
Thus the change from a star to a sharp records the change from their
parameter $\theta$ to the standard Jack parameter $\alpha=1/\theta$.
The polynomial $P_\mu^\#$ is characterized by the interpolation
conditions
\begin{equation}\label{eq:shifted-Jack-vanishing}
  P_\mu^\#(\lambda;\alpha)=0
  \qquad\text{unless }\mu\subseteq\lambda,
\end{equation}
together with highest homogeneous component $P_\mu^{(\alpha)}$ and
the normalization
\begin{equation}\label{eq:Psharp-diagonal}
  P_\mu^\#(\mu;\alpha)
  =\alpha^{-|\mu|}c'_\mu(\alpha).
\end{equation}
The corresponding integral normalization is
$J_\mu^\#=c_\mu(\alpha)P_\mu^\#$.  Equivalently, the binomial
characterization used in \cite{OOJack,AF} gives, for $\lambda\vdash n$
and $m=|\mu|$,
\begin{equation}\label{eq:shifted-Jack-coefficient-formula}
 \alpha^mP_\mu^\#(\lambda;\alpha)
 =\frac{(n)_m c_\lambda(\alpha)}{n!}
  \scalar{P_\lambda^{(\alpha)}}
  {P'_\mu p_1^{n-m}}_1.
\end{equation}

\begin{theorem}\label{thm:stable-Pprime-spectrum}
Let $\mu\vdash m$ and $\lambda\vdash n$.  Then
\begin{equation}\label{eq:stable-Pprime-spectrum}
  \widehat P'_\mu\times_\alpha P'_\lambda
  =\alpha^mP_\mu^\#(\lambda;\alpha)P'_\lambda.
\end{equation}
Consequently, the stable span of the $\widehat P'_\mu$ is a
commutative algebra for $\times_\alpha$, and the correspondence
\begin{equation}\label{eq:stable-shifted-Jack-map}
  \widehat P'_\mu
  \longmapsto
  \alpha^{|\mu|}P_\mu^\#
\end{equation}
identifies it with the algebra of shifted Jack functions.
\end{theorem}

\begin{proof}
The degree-$n$ component of $\widehat P'_\mu$ is
$(n)_mP'_\mu p_1^{n-m}$.  Its coefficient on $Q'_\lambda$ is, by
Hall duality,
\begin{equation}
 (n)_m
 \scalar{P_\lambda^{(\alpha)}}{P'_\mu p_1^{n-m}}_1.
\end{equation}
Formula \eqref{eq:GJ-general-eigenvalue} therefore shows that its
eigenvalue on $P'_\lambda$ is
\begin{equation}
 \frac{(n)_m c_\lambda(\alpha)}{n!}
 \scalar{P_\lambda^{(\alpha)}}{P'_\mu p_1^{n-m}}_1.
\end{equation}
By \eqref{eq:shifted-Jack-coefficient-formula}, this is
$\alpha^mP_\mu^\#(\lambda;\alpha)$, proving
\eqref{eq:stable-Pprime-spectrum}.  The last assertion follows by
multiplying eigenvalues, exactly as in
\cref{prop:stable-shifted-isomorphism}.
\end{proof}

Thus shifted Jack functions arise here for the same reason as shifted
Schur functions: they are the eigenvalues of stable multiplication
operators.  Notice that the factor $\alpha^m$ in
\eqref{eq:stable-Pprime-spectrum} is forced by the plethystic
normalization \eqref{eq:Pprime-definition}.

\subsection{The Jack Pieri rule}

Write the ordinary Jack Pieri rule in the $P'$ normalization as
\begin{equation}\label{eq:Pprime-Pieri}
  P'_\mu P'_{(1)}
  =\sum_{\nu/\mu=\square}
    \psi'_{\nu/\mu}(\alpha)P'_\nu.
\end{equation}
In view of \eqref{eq:Pprime-one}, this is equivalent to
\begin{equation}
  \alpha P'_\mu p_1
  =\sum_{\nu/\mu=\square}
    \psi'_{\nu/\mu}(\alpha)P'_\nu.
\end{equation}

\begin{theorem}[Stable Jack Pieri rule]
\label{thm:stable-Jack-Pieri}
For $\mu\vdash m$,
\begin{equation}\label{eq:stable-Jack-Pieri}
 \widehat P'_\mu\times_\alpha\widehat P'_{(1)}
 =\alpha m\widehat P'_\mu
  +\sum_{\nu/\mu=\square}
    \psi'_{\nu/\mu}(\alpha)\widehat P'_\nu.
\end{equation}
Equivalently, the shifted Jack polynomials satisfy
\begin{equation}\label{eq:shifted-Jack-Pieri}
 P_\mu^\#P_{(1)}^\#
 =mP_\mu^\#
  +\sum_{\nu/\mu=\square}
    \psi'_{\nu/\mu}(\alpha)P_\nu^\#.
\end{equation}
\end{theorem}

\begin{proof}
The degree-$N$ component of $\widehat P'_{(1)}$ is
$\alpha Np_1^N$.  Since $p_1^N$ is the identity for
$\times_\alpha$ on $\Lambda_N$, one obtains
\begin{align}
\widehat P'_\mu\times_\alpha\widehat P'_{(1)}
 &=\alpha\sum_{k\geq0}(m+k)
   \frac{(m+k)!}{k!}P'_\mu p_1^k\notag\\
 &=\alpha m\widehat P'_\mu
   +\alpha\sum_{k\geq1}
    \frac{(m+k)!}{(k-1)!}
    (P'_\mu p_1)p_1^{k-1}.
\end{align}
Apply \eqref{eq:Pprime-Pieri} to the second term to obtain
\eqref{eq:stable-Jack-Pieri}.  Finally, apply both sides to
$P'_\lambda$ and use \eqref{eq:stable-Pprime-spectrum}; after
dividing by $\alpha^{m+1}$, one obtains
\eqref{eq:shifted-Jack-Pieri}.
\end{proof}

\subsection{The Alexandersson--F\'eray recurrence}

Define the shifted Jack structure constants by
\begin{equation}\label{eq:shifted-Jack-structure-constants}
  P_\mu^\#P_\nu^\#
  =\sum_\lambda
    c_{\mu\nu}^{\lambda}(\alpha)P_\lambda^\#.
\end{equation}
The same associativity argument as in
\cref{thm:molev-sagan-recurrence} gives the Jack deformation of the
Molev--Sagan recurrence.

\begin{theorem}[Alexandersson--F\'eray]
\label{thm:Alexandersson-Feray}
Let $\nu\vdash n$ and $\lambda\vdash l$, with $l>n$.  Then
\begin{align}
 (l-n)c_{\mu\nu}^{\lambda}(\alpha)
 &={}
 \sum_{\nu^+/\nu=\square}
  \psi'_{\nu^+/\nu}(\alpha)
  c_{\mu\nu^+}^{\lambda}(\alpha)
 \notag\\
 &\quad-
 \sum_{\lambda/\lambda^-=\square}
  \psi'_{\lambda/\lambda^-}(\alpha)
  c_{\mu\nu}^{\lambda^-}(\alpha).
\label{eq:Alexandersson-Feray}
\end{align}
The initial conditions are
\begin{equation}\label{eq:AF-initial}
 c_{\mu\nu}^{\lambda}(\alpha)
 =\delta_{\lambda\nu}P_\mu^\#(\nu;\alpha)
 \qquad (|\lambda|=|\nu|).
\end{equation}
The recurrence and the initial conditions determine all the
coefficients.
\end{theorem}

\begin{proof}
Multiply \eqref{eq:shifted-Jack-structure-constants} by
$P_{(1)}^\#$ and use \eqref{eq:shifted-Jack-Pieri}.  Expanding first on
the $\nu$-factor, the coefficient of $P_\lambda^\#$ is
\begin{equation}
 n c_{\mu\nu}^{\lambda}
 +\sum_{\nu^+/\nu=\square}
   \psi'_{\nu^+/\nu}c_{\mu\nu^+}^{\lambda}.
\end{equation}
Expanding instead after the product
$P_\mu^\#P_\nu^\#$, the same coefficient is
\begin{equation}
 l c_{\mu\nu}^{\lambda}
 +\sum_{\lambda/\lambda^-=\square}
   \psi'_{\lambda/\lambda^-}c_{\mu\nu}^{\lambda^-}.
\end{equation}
Equating the two expressions gives
\eqref{eq:Alexandersson-Feray}.  The initial values follow from the
vanishing property \eqref{eq:shifted-Jack-vanishing}.  As before, if
$d=|\lambda|-|\nu|>0$, every coefficient on the right-hand side has
difference $d-1$, so induction on $d$ proves uniqueness.
\end{proof}

At $\alpha=1$, $P'_\lambda=s_\lambda$,
$\psi'_{\nu/\mu}(1)=1$, and the results of this section reduce to the
shifted Schur formulas of \cref{sec:shifted-symmetric-functions}.
The next step is to express specific stable Jack series, notably the
power-sum series $\widehat p_\mu$, by differential operators and to
compare their eigenvalues with Jack characters.

\section{Stable power sums and differential operators}
\label{sec:jack-differential-operators}

We now return to the stable power-sum series
\begin{equation}\label{eq:stable-power-sum-jack}
  \widehat p_\mu
  =\frac{m!p_\mu}{(1-p_1)^{m+1}},
  \qquad \mu\vdash m,
\end{equation}
and study the differential operators representing multiplication by
these series for the Goulden--Jackson product.  Unlike the stable
$P'$-series of the preceding section, their eigenvalues are Jack
characters.  This gives a direct operator realization of the algebra
of polynomial functions on Young diagrams.

For every partition $\mu$, we denote by $\Delta_\mu(\alpha)$ the
operator of multiplication by $\widehat p_\mu$ for the deformed
product:
\begin{equation}\label{eq:Delta-mu-definition}
  \Delta_\mu(\alpha)f
  =\widehat p_\mu\times_\alpha f.
\end{equation}
When $\mu=(r)$ has a single part, we simply write
\begin{equation}\label{eq:Delta-r-notation}
  \Delta_r(\alpha)=\Delta_{(r)}(\alpha).
\end{equation}
Thus the family $\{\Delta_\mu(\alpha)\}$ is a commuting family, and
\eqref{eq:stable-power-sum-spectrum} gives its joint spectrum on the
$Q'$-basis.

\subsection{Jack characters as eigenvalues}

Let $\mathrm{Ch}_\mu^{(\alpha)}(\lambda)$ denote the Jack character in the
normalization for which
\begin{equation}\label{eq:stable-power-sum-spectrum}
  \widehat p_\mu\times_\alpha Q'_\lambda
  =\mathrm{Ch}_\mu^{(\alpha)}(\lambda)Q'_\lambda.
\end{equation}
This formula may equivalently be taken as the normalization of the
Jack characters used here.  Indeed, on the homogeneous component of
degree $n\geq m$, one has
\begin{equation}
  (\widehat p_\mu)_n=(n)_m p_{\mu1^{n-m}}.
\end{equation}
Expanding this expression in the $Q'$-basis and applying
\eqref{eq:GJ-general-eigenvalue} gives
\eqref{eq:stable-power-sum-spectrum} directly.

\begin{proposition}\label{prop:Jack-character-algebra}
The span of the $\widehat p_\mu$ is closed under
$\times_\alpha$.  If
\begin{equation}\label{eq:stable-power-sum-product}
  \widehat p_\mu\times_\alpha\widehat p_\nu
  =\sum_\rho d_{\mu\nu}^{\rho}(\alpha)\widehat p_\rho,
\end{equation}
then the coefficients $d_{\mu\nu}^{\rho}(\alpha)$ are the structure
constants of Jack characters:
\begin{equation}\label{eq:Jack-character-product}
 \mathrm{Ch}_\mu^{(\alpha)}\mathrm{Ch}_\nu^{(\alpha)}
 =\sum_\rho d_{\mu\nu}^{\rho}(\alpha)
  \mathrm{Ch}_\rho^{(\alpha)}.
\end{equation}
\end{proposition}

\begin{proof}
Apply both sides of \eqref{eq:stable-power-sum-product} to
$Q'_\lambda$.  By associativity and
\eqref{eq:stable-power-sum-spectrum}, the left-hand side acts by
$\mathrm{Ch}_\mu^{(\alpha)}(\lambda)\mathrm{Ch}_\nu^{(\alpha)}(\lambda)$, whereas
the right-hand side acts by
$\sum_\rho d_{\mu\nu}^{\rho}(\alpha)
\mathrm{Ch}_\rho^{(\alpha)}(\lambda)$.  The Jack characters are linearly
independent, and the assertion follows.
\end{proof}

Thus the stable power sums play for Jack characters the role played
by the $p_\mu^\#$ in the classical shifted-symmetric theory.  In
particular, stability of the Jack-character algebra is equivalent
to the existence of the stable series \eqref{eq:stable-power-sum-jack}.

\subsection{The first deformed cut-and-join operator}

We write
\begin{equation}
  D_{p_k}=k\frac{\partial}{\partial p_k},
  \qquad
  D_{p_\lambda}=\prod_iD_{p_{\lambda_i}}.
\end{equation}
By definition, $\Delta_2(\alpha)=\Delta_{(2)}(\alpha)$ represents
multiplication by $\widehat p_2$.  It is the deformed cut-and-join
operator of Goulden and Jackson \cite{GJ96}:
\begin{equation}\label{eq:Delta2-alpha}
\begin{split}
 \Delta_2(\alpha)
 ={}&\sum_{i,j\geq1}
 \left(
   \alpha p_ip_jD_{p_{i+j}}
   +p_{i+j}D_{p_i}D_{p_j}
 \right)\\
 &+(\alpha-1)\sum_{k\geq1}(k-1)p_kD_{p_k}.
\end{split}
\end{equation}
The first line is the usual join--cut operator, with a factor
$\alpha$ on the joining term.  The second line is genuinely
deformed: it vanishes at $\alpha=1$ and records the contribution that
cannot be represented by permutations alone.

For example, applying \eqref{eq:Delta2-alpha} to
$\widehat p_2=2p_2/(1-p_1)^3$ gives
\begin{equation}\label{eq:jack-p2-times-p2}
 \widehat p_2\times_\alpha\widehat p_2
 =\widehat p_{22}+4\widehat p_3
  +2\alpha\widehat p_{11}
  +2(\alpha-1)\widehat p_2.
\end{equation}
Similarly,
\begin{align}
 \Delta_2(\alpha)\widehat p_3
 &=\widehat p_{32}+6\widehat p_4
   +6\alpha\widehat p_{21}
   +6(\alpha-1)\widehat p_3,
 \label{eq:Delta2-p3}\\
 \Delta_2(\alpha)\widehat p_4
 &=\widehat p_{42}+8\widehat p_5+8\widehat p_{31}
   +4\alpha\widehat p_{22}
   +12(\alpha-1)\widehat p_4.
 \label{eq:Delta2-p4}
\end{align}
These identities provide elementary checks of both the stable basis
and the normalization of the deformed product.

\subsection{Elementary stable series and the operators of
Nazarov--Sklyanin}

For a symmetric function $f$ of degree $m$, set
\begin{equation}\label{eq:check-f}
  \check f
  =\frac{m!f[\alpha X]}{(1-p_1)^{m+1}}.
\end{equation}
In particular, $\check e_k=\widehat P'_{(1^k)}$.  The operators
associated with the $\check e_k$ are the adjoints of the
Sekiguchi--Debiard operators at infinity studied by Nazarov and
Sklyanin \cite{NS}.  In our normalization they are
\begin{equation}\label{eq:NS-Bk}
 B_k
 =\sum_{\ell(\lambda)=k}
   \widetilde m_\lambda[\alpha X]D_{m_\lambda},
\end{equation}
where
$\widetilde m_\lambda=(\prod_i m_i(\lambda)!)m_\lambda$ and
$D_f$ denotes the adjoint of multiplication by $f$ for the ordinary
Hall scalar product.  Equivalently,
\begin{equation}\label{eq:Bk-stable-elementary}
 B_k f=\check e_k\times_\alpha f.
\end{equation}

Indeed, multiplication by $\check e_k$ has
eigenvalue $\alpha^kP_{(1^k)}^\#(\lambda;\alpha)$ on the
$P'_\lambda$-eigenline. The Nazarov--Sklyanin eigenvalue formula gives
the same eigenvalue for $(-1)^kA^{(k)\perp}$:
\[
 (-1)^kA^{(k)\perp}P'_\lambda
 =\alpha^kP_{(1^k)}^\#(\lambda;\alpha)P'_\lambda.
\]
Since the $P'_\lambda$ form a basis, the two operators coincide.
Thus the operator associated with $\check e_k$ is
\[
 B_k=(-1)^kA^{(k)\perp}.
\]

This formula is strikingly close to its classical specialization.
The Jack deformation is absorbed by the plethystic substitution
$X\mapsto\alpha X$; no additional correction term is needed in
\eqref{eq:NS-Bk}.  By contrast, the power-sum operator
\eqref{eq:Delta2-alpha} contains the diagonal correction proportional
to $\alpha-1$.  Newton identities explain how such terms appear when
one passes from the elementary generators to power sums.

For instance, stabilization of the elementary identities gives
\begin{align}
 \check e_{11}&=\check e_1\times_\alpha\check e_1
                 -\alpha\check e_1,
 \label{eq:e11-stable}\\
 \check e_{21}&=\check e_2\times_\alpha\check e_1
                 -2\alpha\check e_2,
 \label{eq:e21-stable}\\
 \check e_{111}&=\check e_1^{\times_\alpha3}
 -3\alpha\check e_1^{\times_\alpha2}
 +2\alpha^2\check e_1.
 \label{eq:e111-stable}
\end{align}
Since
$p_2=e_{11}-2e_2$ and
$p_3=e_{111}-3e_{21}+3e_3$, the multiplication operators associated
with the first power sums are consequently
\begin{align}
 \alpha\Delta_2(\alpha)
 &=B_1^2-\alpha B_1-2B_2,
 \label{eq:Delta2-B}\\
 \alpha\Delta_3(\alpha)
 &=B_1^3-3B_2B_1+3B_3
   -3\alpha B_1^2+6\alpha B_2+2\alpha^2B_1.
 \label{eq:Delta3-B}
\end{align}
Here products denote composition of operators.  Formula
\eqref{eq:Delta3-B}, together with the explicit expression
\eqref{eq:NS-Bk}, reduces the determination of $\Delta_3(\alpha)$ to
a finite application of the Leibniz rule.  The same method applies
to $p_4$, although the resulting differential operator is already
substantially larger.  We carry out these expansions in the next
section.

\section{The operators $\Delta_3(\alpha)$ and $\Delta_4(\alpha)$}
\label{sec:low-degree-jack-operators}

We now expand the operator identities obtained at the end of
\cref{sec:jack-differential-operators}.  The calculations are most
convenient in a basis of normally ordered differential monomials.
They reveal a uniform leading part, inherited from permutations, and
additional terms divisible by $\alpha-1$.

\subsection{Composition of differential monomials}

For two partitions, or more generally two finite multisets of
positive integers, write
\begin{equation}\label{eq:I-J-notation}
 (I\mid J)=p_I D_{p_J}.
\end{equation}
If $A$ is a submultiset of both $J$ and $K$, let $J\setminus A$ and
$K\setminus A$ denote multiset difference.  The Leibniz rule gives
the following finite composition formula.

\begin{lemma}\label{lem:differential-monomial-composition}
For all finite multisets $I,J,K,L$ of positive integers,
\begin{equation}\label{eq:differential-monomial-composition}
 (I\mid J)(K\mid L)
 =\sum_{A\subseteq J\cap K}
 \frac{z_Jz_K}
 {z_Az_{J\setminus A}z_{K\setminus A}}
 \bigl(I\cup(K\setminus A)
 \mid(J\setminus A)\cup L\bigr).
\end{equation}
Here the sum is over submultisets and
$z_\lambda=\prod_r r^{m_r(\lambda)}m_r(\lambda)!$.
\end{lemma}

\begin{proof}
For each $r$, choose $a_r$ of the $m_r(J)$ derivations in $p_r$ to
act on $p_K$.  Their contribution is
\begin{equation}
 \binom{m_r(J)}{a_r}
 \bigl(m_r(K)\bigr)_{a_r}r^{a_r}.
\end{equation}
The product of these quantities over $r$ is precisely the quotient
of $z$-factors in \eqref{eq:differential-monomial-composition}.
\end{proof}

This lemma turns every polynomial expression in the operators $B_k$
into a normally ordered differential operator.  Since only finitely
many contractions are possible in each product, the calculation is
finite at every fixed degree.

\subsection{The third operator}

The expansion suggested by \eqref{eq:Delta3-B} has the following
form.  We record it as a computational formula: it follows by a
finite use of \cref{lem:differential-monomial-composition}, and has
also been checked directly on homogeneous components in low degrees.

\begin{conjecture}\label{conj:Delta3-explicit}
The operator $\Delta_3(\alpha)$ is
\begin{align}
\Delta_3(\alpha)
={}&\sum_{i,j,k\geq1}\Bigl(
 p_{i+j+k}D_{p_i}D_{p_j}D_{p_k}
 +\alpha p_{i+k}p_jD_{p_{i+j}}D_{p_k}
 \notag\\
&\qquad
 +\alpha p_{i+j}p_kD_{p_i}D_{p_{j+k}}
 +p_{i+j+k}D_{p_{i+j+k}}
 \notag\\
&\qquad
 +\alpha^2p_ip_jp_kD_{p_{i+j+k}}
 +\alpha p_ip_{j+k}D_{p_{i+k}}D_{p_j}
 \Bigr)
 \notag\\
&+(\alpha-1)\Biggl[
 \frac32\sum_{i,j\geq1}(i+j-2)
 \Bigl(
  \alpha p_ip_jD_{p_{i+j}}
  +p_{i+j}D_{p_i}D_{p_j}
 \Bigr)
 \notag\\
&\hspace{38mm}
 +(2\alpha-1)\sum_{k\geq1}
 \binom{k-1}{2}p_kD_{p_k}
 \Biggr].
\label{eq:Delta3-explicit}
\end{align}
\end{conjecture}

The first three lines of \eqref{eq:Delta3-explicit} specialize at
$\alpha=1$ to the permutation formula
\eqref{eq:A-mu-explicit} for $\mu=(3)$.  We call this the
\emph{top part} and denote its $\alpha$-deformation by
$\Delta_3^{\mathrm{top}}(\alpha)$.  Thus
\begin{equation}\label{eq:Delta3-top-correction}
 \Delta_3(\alpha)
 =\Delta_3^{\mathrm{top}}(\alpha)
  +(\alpha-1)\Delta'_3(\alpha),
\end{equation}
where $\Delta'_3(\alpha)$ is the expression in square brackets.
The correction contains terms with only two, or only one, effective
indices.  This lowering of the number of indices is the operator
counterpart of contractions that are not encoded by permutations.

\subsection{The fourth operator}

The expression for $\Delta_4$ is considerably longer.  It becomes
manageable if the top part is defined from the classical operator
$\mathcal A_{(4)}$ of \eqref{eq:A-mu-explicit}: in every monomial
$(I\mid J)$ of $\mathcal A_{(4)}$, multiply the coefficient by
$\alpha^{\ell(I)-1}$.  We denote the resulting operator by
$\Delta_4^{\mathrm{top}}(\alpha)$.

\begin{conjecture}\label{conj:Delta4-explicit}
One has
\begin{equation}\label{eq:Delta4-top-correction}
 \Delta_4(\alpha)
 =\Delta_4^{\mathrm{top}}(\alpha)
  +(\alpha-1)\Delta'_4(\alpha),
\end{equation}
where
\begin{align}
\Delta'_4(\alpha)
={}&\alpha\sum_{i,j,k,l\geq1}
 (i+j,k+l\mid i+l,j+k)
 \label{eq:Delta4-correction-22}\\
&+2\sum_{i,j,k\geq1}(i+j+k-3)
 \bigl[
  \alpha^2(i,j,k\mid i+j+k)
  +(i+j+k\mid i,j,k)
 \bigr]
 \label{eq:Delta4-correction-31}\\
&+4\alpha\sum_{i,j,k\geq1}(2i+j+k-4)
 (i+j,k\mid i+k,j)
 \label{eq:Delta4-correction-mixed}\\
&+\sum_{i,j\geq1}F_{ij}(\alpha)
 \bigl[
  \alpha(i,j\mid i+j)+(i+j\mid i,j)
 \bigr]
 \label{eq:Delta4-correction-21}\\
&+(6\alpha^2-5\alpha+6)
 \sum_{k\geq1}\binom{k-1}{3}(k\mid k),
 \label{eq:Delta4-correction-11}
\end{align}
with
\begin{equation}\label{eq:Fij-alpha}
\begin{split}
 F_{ij}(\alpha)
 ={}&\alpha\left(
 4\binom{i-1}{2}+4\binom{j-1}{2}
 +3(i-1)(j-1)
 \right)\\
 &-(i+j-2)(i+j-3).
\end{split}
\end{equation}
\end{conjecture}

The decomposition \eqref{eq:Delta4-top-correction} has the same
structural meaning as \eqref{eq:Delta3-top-correction}.  The top part
is obtained from the six-- or twenty-four--term permutation mechanism
of the classical theory, whereas $\Delta'_4$ consists of additional
contractions.  The latter vanish from the full operator at
$\alpha=1$, but survive in the real and quaternionic specializations
$\alpha=2$ and $\alpha=1/2$.

\subsection{Status of the formulas}

The ingredients leading to
Conjectures \ref{conj:Delta3-explicit} and \ref{conj:Delta4-explicit} are: the
Newton identities, the
Nazarov--Sklyanin formula \eqref{eq:NS-Bk}, and the composition rule
\eqref{eq:differential-monomial-composition}.  Consequently, both
statements can in principle be proved by a finite normal-ordering
calculation.  We retain the conjectural label because our present
derivation was computer-assisted and has not yet been converted into
a complete term-by-term proof.

This status is stronger than mere interpolation in the parameter:
the formulas were obtained in the differential-operator algebra and
not reconstructed from the three specializations
$\alpha=1,2,1/2$.  Later, matrix-integral calculations at
$\alpha=2$ and $\alpha=1/2$ will provide an independent explanation
of the correction terms.


\section{Real Gaussian matrix integrals}
\label{sec:real-matrix-integrals}

We now give a matrix-integral interpretation of the specialization
$\alpha=2$.  Since several normalizations of zonal polynomials and
Gaussian measures coexist in the literature, we begin directly with
the Gaussian model of Hanlon, Stanley, and Stembridge \cite{HSS}; see
also the combinatorial proof of Goulden and Jackson \cite{GJ95}.
This avoids importing normalization factors from the integral
formulas for zonal polynomials in \cite[Chapter VII]{Macdonald}.

\subsection{The real Gaussian model}

Let $M_N(\mathbb R)$ be equipped with the probability measure
\begin{equation}\label{eq:real-Gaussian-measure}
 d\nu(Z)
 =(2\pi)^{-N^2/2}
 \exp\left(-\frac12\operatorname{tr}(ZZ^t)\right)dZ,
 \qquad
 dZ=\prod_{i,j=1}^N dz_{ij}.
\end{equation}
Thus the entries of $Z$ are independent standard real Gaussian
variables and
\begin{equation}\label{eq:real-Gaussian-covariance}
 \mathbb E(z_{ij}z_{kl})=\delta_{ik}\delta_{jl}.
\end{equation}
This is the normalization used in the normally distributed matrix
model of \cite{HSS,GJ95}.  In particular, there is no extra factor in
the covariance.

Let $A$ and $B$ be real symmetric $N\times N$ matrices.  By
orthogonal invariance of \eqref{eq:real-Gaussian-measure}, we may
assume that
\begin{equation}
 A=\operatorname{diag}(a_1,\ldots,a_N),
 \qquad
 B=\operatorname{diag}(b_1,\ldots,b_N).
\end{equation}
For a partition $\lambda\vdash n$, choose a permutation
$\sigma\in\Sym_n$ of cycle type $\lambda$.  In the expansion of
$p_\lambda(AZBZ^t)$, label the two occurrences of $Z$ in the $k$th
factor by $k$ and $\bar k$.  Introduce the two fixed matchings
\begin{align}
 \epsilon
 &=1\bar1\mid2\bar2\mid\cdots\mid n\bar n,
 \label{eq:epsilon-matching}\\
 \delta_\lambda
 &=\bar1\,\sigma(1)\mid\bar2\,\sigma(2)
   \mid\cdots\mid\bar n\,\sigma(n).
 \label{eq:delta-lambda-matching}
\end{align}
If $\delta_1$ and $\delta_2$ are perfect matchings, their union is a
disjoint union of even alternating cycles.  We denote by
$\Lambda(\delta_1,\delta_2)$ the partition formed by their
half-lengths.

Figure~\ref{fig:real-Wick-alternating-cycles} shows this construction
for one matching in degree three.
\begin{figure}[H]
\centering
\begin{tikzpicture}[
  x=1cm,y=1cm,
  vertex/.style={circle,draw,fill=white,inner sep=1.4pt,font=\small},
  wick/.style={red!75!black,very thick},
  eps/.style={gray!75,very thick,dashed},
  cyc/.style={blue!70!black,very thick,densely dotted}
]
\def\panelA{0}
\def\panelB{7.0}

\foreach \x/\lab in {0/1,1.7/2,3.4/3}{
  \node[vertex] (uA\lab) at ({\panelA+\x},1.55) {$\lab$};
  \node[vertex] (bA\lab) at ({\panelA+\x},0) {$\bar\lab$};
}
\draw[eps] (uA1)--(bA1);
\draw[eps] (uA2)--(bA2);
\draw[eps] (uA3)--(bA3);
\draw[wick] (uA1) to[bend left=16] (uA2);
\draw[wick] (bA1) to[bend left=10] (uA3);
\draw[wick] (bA2) to[bend right=16] (bA3);
\node[font=\small] at ({\panelA+1.7},-0.62)
 {$\delta\cup\epsilon:\quad B(\delta)=(3)$};

\foreach \x/\lab in {0/1,1.7/2,3.4/3}{
  \node[vertex] (uB\lab) at ({\panelB+\x},1.55) {$\lab$};
  \node[vertex] (bB\lab) at ({\panelB+\x},0) {$\bar\lab$};
}
\draw[cyc] (bB1)--(uB2);
\draw[cyc] (bB2)--(uB3);
\draw[cyc] (bB3) to[bend left=28] (uB1);
\draw[wick] (uB1) to[bend left=16] (uB2);
\draw[wick] (bB1) to[bend left=10] (uB3);
\draw[wick] (bB2) to[bend right=16] (bB3);
\node[font=\small] at ({\panelB+1.7},-0.62)
 {$\delta\cup\delta_{(3)}:\quad A(\delta)=(3)$};

\draw[wick] (1.7,2.28)--(2.35,2.28)
 node[right,black,font=\small] {$\delta$};
\draw[eps] (4.15,2.28)--(4.8,2.28)
 node[right,black,font=\small] {$\epsilon$};
\draw[cyc] (7.0,2.28)--(7.65,2.28)
 node[right,black,font=\small] {$\delta_{(3)}$};
\end{tikzpicture}
\caption{A Wick matching and its two alternating graphs.  Here
$\delta=12\mid\bar1 3\mid\bar2\bar3$,
$\epsilon=1\bar1\mid2\bar2\mid3\bar3$, and
$\delta_{(3)}=\bar1 2\mid\bar2 3\mid\bar3 1$.  Each union is one
alternating cycle of length six, hence both half-length partitions are
$(3)$.  In general, the components of the left and right graphs give
$B(\delta)$ and $A(\delta)$, respectively.}
\label{fig:real-Wick-alternating-cycles}
\end{figure}
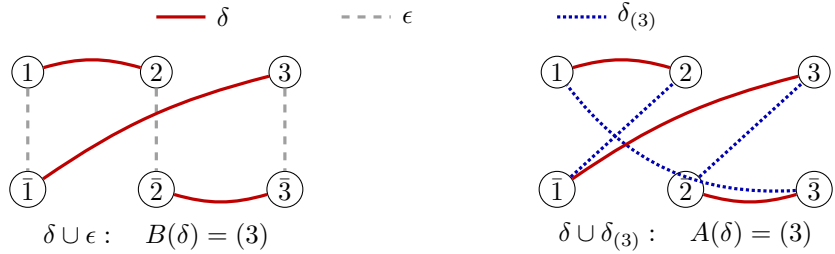

\begin{theorem}[Hanlon--Stanley--Stembridge]
\label{thm:HSS-real-Wick}
With the preceding normalization,
\begin{equation}\label{eq:real-Wick-matching-formula}
 \mathbb E\,p_\lambda(AZBZ^t)
 =\sum_{\delta\in\mathcal M_{2n}}
 p_{\Lambda(\delta,\delta_\lambda)}(A)
 p_{\Lambda(\delta,\epsilon)}(B),
\end{equation}
where $\mathcal M_{2n}$ is the set of perfect matchings of
$\{1,\bar1,\ldots,n,\bar n\}$.
Consequently,
\begin{equation}\label{eq:HSS-Jack-connection-coefficients}
 \mathbb E\,p_\lambda(AZBZ^t)
 =\sum_{\mu,\nu\vdash n}
 a_{\mu\nu}^{\lambda}(2)p_\mu(A)p_\nu(B),
\end{equation}
where, in our normalization, $a_{\mu\nu}^{\lambda}(2)$ is the
number of matchings $\delta$ satisfying
\begin{equation}
 \Lambda(\delta,\delta_\lambda)=\mu,
 \qquad
 \Lambda(\delta,\epsilon)=\nu.
\end{equation}
\end{theorem}

\begin{proof}
The $2n$ matrix entries in the expansion of
$p_\lambda(AZBZ^t)$ are paired by Wick's formula.  A Wick matching
$\delta$ identifies the column indices along the components of
$\delta\cup\epsilon$ and the row indices along the components of
$\delta\cup\delta_\lambda$.  Summing the free indices gives the two
power sums in \eqref{eq:real-Wick-matching-formula}.  Grouping
matchings by their two types proves
\eqref{eq:HSS-Jack-connection-coefficients}.
\end{proof}

The last formula is the $\alpha=2$ coproduct underlying the
Goulden--Jackson product.  
The coefficients are fixed by the matching formula itself, and do not
depend of the  normalization of zonal polynomials.

\subsection{Two normalization checks}

For $\lambda=(2)$ there are three Wick matchings, and
\begin{equation}\label{eq:real-p2-moment}
\begin{split}
 \mathbb E\,p_2(AZBZ^t)
 ={}&p_2(A)p_{11}(B)+p_{11}(A)p_2(B)\\
 &+p_2(A)p_2(B).
\end{split}
\end{equation}
The first two terms are contributed by bipartite matchings, that is,
matchings whose edges join barred to unbarred vertices.  They are the
two matchings encoded by permutations in $\Sym_2$.  The last term is
contributed by
\begin{equation}
 12\mid\bar1\bar2,
\end{equation}
and has no permutation analogue.  It is the first source of the
correction proportional to $\alpha-1$ in the differential operators.

At degree four, direct enumeration of the $7!!=105$ matchings gives
the following symmetric matrix.  Rows record
$\Lambda(\delta,\delta_{(4)})$ and columns record
$\Lambda(\delta,\epsilon)$:
\begin{equation}\label{eq:real-degree-four-table}
\begin{array}{c|rrrrr}
 &1111&211&22&31&4\\ \hline
1111&0&0&0&0&1\\
211 &0&0&2&4&6\\
22  &0&2&1&4&5\\
31  &0&4&4&8&16\\
4   &1&6&5&16&20
\end{array}
\end{equation}
The entries sum to $105$.  As a further check, if $q(\delta)$ is the
number of edges joining two unbarred vertices, the matchings split as
\begin{equation}\label{eq:real-q-distribution}
 \#\{q=0\}=24,
 \qquad
 \#\{q=1\}=72,
 \qquad
 \#\{q=2\}=9.
\end{equation}
The $24=4!$ matchings with $q=0$ are precisely the bipartite ones.

\subsection{The Cauchy deformation and the propagator}

To pass from fixed moments to stable series, consider
\begin{equation}\label{eq:real-deformed-integral}
 \int_{M_N(\mathbb R)}
 p_\rho(AZBZ^t)
 \exp\left(\frac12p_1(AZBZ^t)\right)d\nu(Z).
\end{equation}
The factor $1/2$ is essential: it is the exponent $1/\alpha$ in the
Jack Cauchy kernel at $\alpha=2$.  For diagonal $A$ and $B$, the
unnormalized density in \eqref{eq:real-deformed-integral} is
\begin{equation}
 (2\pi)^{-N^2/2}
 \exp\left[-\frac12\sum_{i,j}
 (1-a_ib_j)z_{ij}^2\right]dZ.
\end{equation}
Assuming $|a_ib_j|<1$ for convergence, its total mass is
\begin{equation}\label{eq:real-partition-function}
 \mathcal Z_{\mathbb R}(A,B)
 =\prod_{i,j}(1-a_ib_j)^{-1/2}
 =\sigma_1(AB)^{1/2}.
\end{equation}
All identities may alternatively be read as formal power-series
identities in the eigenvalues of $A$ and $B$.

After division by \eqref{eq:real-partition-function}, the modified
Gaussian covariance is
\begin{equation}\label{eq:real-modified-propagator}
 \left\langle\!\left\langle
 z_{ij}z_{kl}
 \right\rangle\!\right\rangle
 =\frac{\delta_{ik}\delta_{jl}}{1-a_ib_j}.
\end{equation}
The geometric expansion of this propagator attaches a positive
integer to every Wick edge.  It is this edge-labeling that produces
the infinite sums occurring in the stable differential operators.

\subsection{Recovery of $\Delta_2(2)$}

Applying Wick's formula with the modified propagator gives
\begin{align}
 \left\langle\!\left\langle
 p_2(AZBZ^t)
 \right\rangle\!\right\rangle
 ={}&\sum_{i,j\geq1}p_{i+j}(A)p_i(B)p_j(B)
 \notag\\
 &+\sum_{i,j\geq1}p_i(A)p_j(A)p_{i+j}(B)
 \notag\\
 &+\sum_{k\geq1}(k-1)p_k(A)p_k(B).
\label{eq:real-stable-p2-coproduct}
\end{align}
As explained in \cite[Eq. (58)]{LT}, such expressions can be converted into differential operators. 
In the passage from this coproduct to our differential-operator
normalization, a monomial $p_I(A)p_J(B)$ receives the factor
$2^{\ell(I)-1}$.  Denote by $\mathcal H^{\mathbb R}_2$ the operator
obtained in this way.  Then \eqref{eq:real-stable-p2-coproduct} yields
\begin{align}
 \sum_{i,j\geq1}\left(
 2p_ip_jD_{p_{i+j}}+p_{i+j}D_{p_i}D_{p_j}
 \right)
 +\sum_{k\geq1}(k-1)p_kD_{p_k}.
\end{align}
By \eqref{eq:Delta2-alpha}, this is exactly
\begin{equation}\label{eq:real-Delta2-identification}
 {\mathcal H^{\mathbb R}_2=\Delta_2(2).}
\end{equation}
Thus the non-bipartite matching already identified in
\eqref{eq:real-p2-moment} accounts precisely for the diagonal
correction in the deformed cut-and-join operator.

\subsection{Recovery of $\Delta_3(2)$}

There are $5!!=15$ perfect matchings on six vertices.  Six of them
are bipartite and are naturally indexed by $\Sym_3$.  Their
contribution is exactly the top part of
\eqref{eq:Delta3-explicit} at $\alpha=2$:
\begin{align}
\Delta_3^{\mathrm{top}}(2)
={}&\sum_{i,j,k\geq1}\bigl[
 (i+j+k\mid i,j,k)
 +2(i+k,j\mid i+j,k)
 \notag\\
&\qquad
 +2(i+j,k\mid i,j+k)
 +(i+j+k\mid i+j+k)
 \notag\\
&\qquad
 +4(i,j,k\mid i+j+k)
 +2(i,j+k\mid i+k,j)
 \bigr].
\label{eq:real-Delta3-top}
\end{align}

The nine non-bipartite matchings split into three families.  After
the edge labels have been summed, their total contribution is
\begin{align}
\mathcal R_3={}&
3\sum_{k\geq1}\binom{k-1}{2}(k\mid k)
\notag\\
&+\frac32\sum_{i,j\geq1}(i+j-2)
\bigl[2(i,j\mid i+j)+(i+j\mid i,j)\bigr].
\label{eq:real-Delta3-correction}
\end{align}
Indeed, the coefficient $\binom{k-1}{2}$ counts compositions of
$k$ into three positive parts.  In the two binary families, fixing
$i$ and $j$ leaves $i+j-2$ possible positions for the internal cut;
the factor $2^{\ell(I)-1}$ accounts for the asymmetry between the two
orientations of the operator.

Evaluating the correction term in
\eqref{eq:Delta3-top-correction} at $\alpha=2$ gives precisely
\eqref{eq:real-Delta3-correction}.  We have therefore proved the
following specialization of \cref{conj:Delta3-explicit}.

\begin{proposition}\label{prop:real-Delta3}
The operator obtained from the real Gaussian integral with insertion
$p_3(AZBZ^t)$ is
\begin{equation}
 \mathcal H^{\mathbb R}_3
 =\Delta_3^{\mathrm{top}}(2)+\mathcal R_3
 =\Delta_3(2).
\end{equation}
Moreover, the top part is contributed by bipartite matchings and the
correction $\mathcal R_3=\Delta'_3(2)$ by non-bipartite matchings.
\end{proposition}

\subsection{Recovery of $\Delta_4(2)$}

For a matching $\delta$ on eight vertices, let $q(\delta)$ be the
number of edges joining two unbarred vertices.  The distribution
\eqref{eq:real-q-distribution} decomposes the integral operator as
\begin{equation}
 \mathcal H^{\mathbb R}_4=X_0+X_1+X_2,
\end{equation}
where $X_q$ is the contribution of matchings with $q(\delta)=q$.
The $24$ matchings contributing to $X_0$ are bipartite, hence
\begin{equation}\label{eq:real-X0}
 X_0=\Delta_4^{\mathrm{top}}(2).
\end{equation}

The nine matchings with $q=2$ give
\begin{align}
X_2={}&
2\sum_{i,j,k,l\geq1}
(i+j,k+l\mid i+l,j+k)
\notag\\
&+2\sum_{i,j\geq1}(i-1)(j-1)
\bigl[2(i,j\mid i+j)+(i+j\mid i,j)\bigr]
\notag\\
&+4\sum_{k\geq1}\binom{k-1}{3}(k\mid k).
\label{eq:real-X2}
\end{align}
The first line comes from the unique matching whose two alternating
graphs both have two components.  This is the real-matrix origin of
the term
\begin{equation}
2\sum_{i,j,k,l\geq1}(i+j,k+l\mid i+l,j+k),
\end{equation}
which cannot arise from a pair of permutations.

For completeness, the $72$ matchings with $q=1$ split according to
the numbers of components on the two sides as follows:
\begin{equation}\label{eq:real-q1-class-table}
\begin{array}{c|rrrrrrrrrr}
\text{class}&1&2&3&4&5&6&7&8&9&10\\ \hline
\text{number}&6&4&8&12&4&2&6&12&2&16\\
2^{\ell(I)-1}\text{-weighted}
&24&8&16&24&8&4&6&12&2&16
\end{array}
\end{equation}
Regrouping these classes by the shape of the resulting differential
monomial gives
\begin{align}
X_1={}&
2\sum_{i,j,k\geq1}(i+j+k-3)
\bigl[4(i,j,k\mid i+j+k)+(i+j+k\mid i,j,k)\bigr]
\notag\\
&+8\sum_{i,j,k\geq1}(2i+j+k-4)
(i+j,k\mid i+k,j)
\notag\\
&+\sum_{i,j\geq1}G_{ij}
\bigl[2(i,j\mid i+j)+(i+j\mid i,j)\bigr]
\notag\\
&+16\sum_{k\geq1}\binom{k-1}{3}(k\mid k),
\label{eq:real-X1}
\end{align}
where
\begin{equation}\label{eq:real-Gij}
 G_{ij}
 =6\binom{i-1}{2}+6\binom{j-1}{2}
  +2(i-1)(j-1).
\end{equation}

Adding \eqref{eq:real-X2} and \eqref{eq:real-X1}, the binary
coefficient becomes
\begin{equation}
G_{ij}+2(i-1)(j-1)=F_{ij}(2),
\end{equation}
where $F_{ij}(\alpha)$ is defined in
\eqref{eq:Fij-alpha}, and the diagonal coefficient becomes
\begin{equation}
16+4=20=6\cdot2^2-5\cdot2+6.
\end{equation}
All five blocks in \eqref{eq:Delta4-correction-22}--
\eqref{eq:Delta4-correction-11} are therefore recovered, term by
term.  This proves the real specialization of
\cref{conj:Delta4-explicit}.

\begin{proposition}\label{prop:real-Delta4}
The operator obtained from the real Gaussian integral with insertion
$p_4(AZBZ^t)$ is
\begin{equation}
 {
 \mathcal H^{\mathbb R}_4
 =X_0+X_1+X_2
 =\Delta_4^{\mathrm{top}}(2)+\Delta'_4(2)
 =\Delta_4(2).}
\end{equation}
The bipartite matchings give the top part and the non-bipartite
matchings give the entire correction.
\end{proposition}

The same enumeration was also checked mechanically by generating all
perfect matchings and all positive edge-labelings of total weight at
most nine.  The comparison is coefficientwise: it involves $49$
operator types for $\Delta'_3(2)$ and $154$ operator types for
$\Delta'_4(2)$, with no discrepancy.


\section{Quaternionic Gaussian matrix integrals}
\label{sec:quaternionic-matrix-integrals}

We now turn to the matrix specialization $\alpha=1/2$.  Since
quaternionic notation is less standard than its real and complex
counterparts, we give all conventions needed below.  In particular, we
will not use a zonal-polynomial integral as a black box: the normalization
will be fixed by elementary Gaussian contractions.

\subsection{Quaternions, matrices, and traces}

Let $\mathbb H$ be the real algebra with basis
$1,\mathbf i,\mathbf j,\mathbf k$ and relations
\[
 \mathbf i^2=\mathbf j^2=\mathbf k^2
 =\mathbf i\mathbf j\mathbf k=-1.
\]
For
$q=x_0+\mathbf i x_1+\mathbf jx_2+\mathbf kx_3$, set
\[
 \bar q=x_0-\mathbf i x_1-\mathbf jx_2-\mathbf kx_3,
 \qquad
 \operatorname{Re}q=\frac{q+\bar q}{2},
 \qquad
 |q|^2=q\bar q=\sum_{r=0}^3x_r^2.
\]
Conjugation reverses products: $\overline{pq}=\bar q\bar p$.

If $Z=(z_{ij})$ is a quaternionic matrix, its adjoint is
$Z^*=\bar Z^{,t}$.  The naive quaternionic trace is not cyclic.  We
therefore use the real trace
\[
 \operatorname{tr}_{\mathbb H}Z
 =\operatorname{Re}\sum_i z_{ii},
\]
which satisfies
$\operatorname{tr}_{\mathbb H}(XY)=\operatorname{tr}_{\mathbb H}(YX)$.
Equivalently, under the standard embedding
\[
 x_0+\mathbf i x_1+\mathbf jx_2+\mathbf kx_3
 \longmapsto
 \begin{pmatrix}
 x_0+\mathrm i x_1&x_2+\mathrm i x_3\\
 -x_2+\mathrm i x_3&x_0-\mathrm i x_1
 \end{pmatrix},
\]
$\operatorname{tr}_{\mathbb H}$ is one half of the ordinary complex
trace.  Thus, if $H$ is quaternionic Hermitian, the quantities
$p_r(H)=\operatorname{tr}_{\mathbb H}(H^r)$ are real and are the power
sums of its real eigenvalues, counted once.

\subsection{A normalized quaternionic Gaussian}

Let
\begin{equation}\label{eq:standard-quaternionic-Gaussian}
 z=\frac12(\xi_0+\mathbf i\xi_1+\mathbf j\xi_2+\mathbf k\xi_3),
\end{equation}
where the $\xi_r$ are independent real $N(0,1)$ variables.  Then
\begin{equation}\label{eq:quaternionic-basic-moments}
 \mathbb E(z\bar z)=1,
 \qquad
 \mathbb E(z^2)=-\frac12.
\end{equation}
For fixed $a,b\in\mathbb H$, expansion in the four real components
gives the two contractions
\begin{align}
 \mathbb E(zazb)&=-\frac12\,\bar a b,
 \label{eq:quaternionic-twisted-contraction}\\
 \mathbb E(za\bar z b)&=\operatorname{Re}(a)b.
 \label{eq:quaternionic-untwisted-contraction}
\end{align}
Indeed, these identities reduce respectively to
\[
 \sum_{e\in\{1,\mathbf i,\mathbf j,\mathbf k\}}eae=-2\bar a,
 \qquad
 \sum_{e\in\{1,\mathbf i,\mathbf j,\mathbf k\}}ea\bar e
 =4\operatorname{Re}(a).
\]
These formulas are the basic reason that quaternionic Wick
contractions retain orientation information: pairing equal rather than
conjugate variables reverses the intervening product and contributes a
factor $-1/2$.

The two local operations are represented in
Figure~\ref{fig:quaternionic-Wick-ribbons}.
\begin{figure}[H]
\centering
\begin{tikzpicture}[
  x=1cm,y=1cm,
  endpoint/.style={circle,fill=black,inner sep=1.5pt},
  band/.style={line width=2.2pt,draw=blue!55},
  boundary/.style={line width=.7pt,draw=blue!85!black},
  orient/.style={-{Stealth[length=2.2mm]},red!75!black,thick}
]
\node[font=\bfseries] at (2.3,2.35) {untwisted contraction};
\node[endpoint] (ul1) at (0.7,1.55) {};
\node[endpoint] (ul2) at (0.7,0.25) {};
\node[endpoint] (ur1) at (3.9,1.55) {};
\node[endpoint] (ur2) at (3.9,0.25) {};
\draw[band] (ul1)--(ur1);
\draw[band] (ul2)--(ur2);
\draw[boundary] (ul1)--(ur1);
\draw[boundary] (ul2)--(ur2);
\draw[orient] (1.15,1.82)--(2.1,1.82);
\draw[orient] (2.1,-0.02)--(1.15,-0.02);
\node[font=\small] at (2.3,-0.55)
 {$z\,a\,\bar z\,b\;\longmapsto\;\operatorname{Re}(a)b$};

\begin{scope}[xshift=6.1cm]
\node[font=\bfseries] at (2.3,2.35) {twisted contraction};
\node[endpoint] (tl1) at (0.7,1.55) {};
\node[endpoint] (tl2) at (0.7,0.25) {};
\node[endpoint] (tr1) at (3.9,1.55) {};
\node[endpoint] (tr2) at (3.9,0.25) {};
\draw[band] (tl1) to[out=0,in=180] (tr2);
\draw[band] (tl2) to[out=0,in=180] (tr1);
\draw[white,line width=3.8pt] (2.03,1.25)--(2.57,0.55);
\draw[boundary] (tl1) to[out=0,in=180] (tr2);
\draw[boundary] (tl2) to[out=0,in=180] (tr1);
\draw[orient] (1.05,1.83)--(2.0,1.45);
\draw[orient] (3.55,1.83)--(2.6,1.45);
\node[font=\small] at (2.3,-0.55)
 {$z\,a\,z\,b\;\longmapsto\;-\frac12\bar a b$};
\end{scope}
\end{tikzpicture}
\caption{The two elementary quaternionic Wick contractions in the
normalization $\mathbb E(z\bar z)=1$.  Pairing conjugate occurrences
preserves the two sides of the ribbon.  Pairing equally oriented
occurrences exchanges them, producing a half-twist, conjugating the
intervening factor, and contributing $-1/2$.  For a complete pairing,
these local ribbons form the M\"obius graph whose Euler characteristic
determines the global weight in
\eqref{eq:Bryc-Pierce-Wick-weight}.}
\label{fig:quaternionic-Wick-ribbons}
\end{figure}
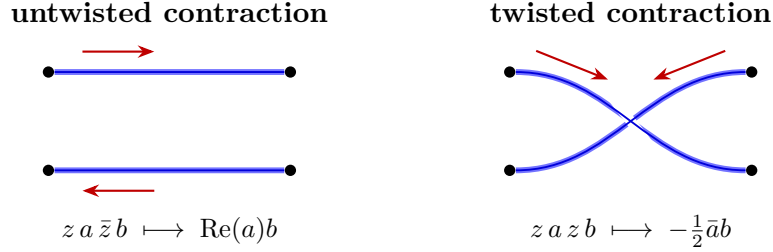

For later reference, the general Wick rule can be encoded by M\"obius
graphs.  Pair the Gaussian occurrences in a product of real traces.
A pair consisting of a variable and its conjugate gives an untwisted
edge, whereas a pair of equally oriented variables gives a twisted
edge.  With the unscaled convention
$Z=\xi_0+\mathbf i\xi_1+\mathbf j\xi_2+\mathbf k\xi_3$, Bryc and
Pierce \cite[Theorem~3.1]{BrycPierce} prove that a M\"obius graph
$\Gamma$ occurring in a product of $m$ traces and containing $n$ edges
has weight
\begin{equation}\label{eq:Bryc-Pierce-Wick-weight}
 4^{,n-m}(-2)^{\chi(\Gamma)},
 \qquad
 \chi(\Gamma)=v(\Gamma)-e(\Gamma)+f(\Gamma),
\end{equation}
with $\chi$ understood additively over connected components.  Formula
\eqref{eq:Bryc-Pierce-Wick-weight} is not an additional normalization
assumption: after rescaling $Z$ by $1/2$, it is the iterated form of
\eqref{eq:quaternionic-twisted-contraction} and
\eqref{eq:quaternionic-untwisted-contraction}.

\subsection{The Gaussian matrix model}

Let $Z=(z_{ij})\in M_N(\mathbb H)$ have independent entries distributed
as in \eqref{eq:standard-quaternionic-Gaussian}.  In real coordinates,
its probability measure is
\begin{equation}\label{eq:quaternionic-Gaussian-measure}
 d\nu_{\mathbb H}(Z)
 =\left(\frac2\pi\right)^{2N^2}
   \exp\{-2\operatorname{tr}_{\mathbb H}(ZZ^*)\}\,dZ,
\end{equation}
where $dZ$ is Lebesgue measure in the $4N^2$ real coordinates.  Thus
\begin{equation}
 \mathbb E(z_{ij}\bar z_{kl})=\delta_{ik}\delta_{jl}.
\end{equation}
Let $A=\operatorname{diag}(a_1,\ldots,a_N)$ and
$B=\operatorname{diag}(b_1,\ldots,b_N)$ be real diagonal matrices and
put
\[
 M=AZBZ^*.
\]
Although $M$ need not itself be Hermitian, cyclicity of the real trace
shows that its power traces agree with those of
$A^{1/2}ZBZ^*A^{1/2}$ whenever the latter expression is defined.  In
particular they are real; formally, we simply define
$p_r(M)=\operatorname{tr}_{\mathbb H}(M^r)$.

As a first check,
\begin{equation}\label{eq:quaternionic-first-moment}
 \mathbb E\,p_1(AZBZ^*)=p_1(A)p_1(B).
\end{equation}
In degree two, expanding the real trace gives
\[
 p_2(M)=\sum_{i,k,j,l}a_ia_kb_jb_l
 z_{ij}\bar z_{kj}z_{kl}\bar z_{il}.
\]
The two bipartite pairings contribute
$p_2(A)p_1(B)^2$ and $p_1(A)^2p_2(B)$.  For the remaining pairing,
\eqref{eq:quaternionic-twisted-contraction} gives, for two independent
normalized quaternions $z,w$,
\[
 \mathbb E(z\bar w z\bar w)=-\frac12.
\]
Consequently
\begin{equation}\label{eq:quaternionic-second-moment}
 \mathbb E\,p_2(AZBZ^*)
 =p_2(A)p_1(B)^2+p_1(A)^2p_2(B)
 -\frac12p_2(A)p_2(B).
\end{equation}
The coefficient of the non-bipartite pairing is therefore
$-1/2=\alpha-1$ at $\alpha=1/2$, in agreement with the real coefficient
$1=\alpha-1$ at $\alpha=2$.

\subsection{Cauchy deformation and recovery of $\Delta_2(1/2)$}

The analogue of \eqref{eq:real-deformed-integral} is obtained by
inserting $\exp\{2p_1(AZBZ^*)\}$.  Assuming $|a_ib_j|<1$, its mass is
\begin{equation}\label{eq:quaternionic-partition-function}
 \mathcal Z_{\mathbb H}(A,B)
 =\prod_{i,j}(1-a_ib_j)^{-2}
 =\sigma_1(AB)^2.
\end{equation}
Thus the three matrix cases obey the uniform formula
$\mathcal Z_\alpha=\sigma_1(AB)^{1/\alpha}$.
After division by \eqref{eq:quaternionic-partition-function}, the
entries remain independent and the contractions become
\begin{align}
 \left\langle\!\left\langle z_{ij}a z_{kl}b
 \right\rangle\!\right\rangle
 &=-\frac12\frac{\delta_{ik}\delta_{jl}}{1-a_ib_j}\,\bar a b,
 \\
 \left\langle\!\left\langle z_{ij}a\bar z_{kl}b
 \right\rangle\!\right\rangle
 &=\frac{\delta_{ik}\delta_{jl}}{1-a_ib_j}\,
   \operatorname{Re}(a)b.
\end{align}
Expanding each propagator geometrically and then using the operator
normalization of Section~\ref{sec:jack-differential-operators} yields
\begin{align}
 \mathcal H^{\mathbb H}_2
 ={}&\sum_{i,j\geq1}
 \left[\frac12(i,j\mid i+j)+(i+j\mid i,j)\right]
 -\frac12\sum_{k\geq1}(k-1)(k\mid k).
 \label{eq:quaternionic-H2}
\end{align}
Comparison with \eqref{eq:Delta2-alpha} gives
\begin{equation}\label{eq:quaternionic-Delta2}
{\mathcal H^{\mathbb H}_2=\Delta_2(1/2).}
\end{equation}

\subsection{The weight of a matching}

We now make the rule needed in higher degree explicit.  For the
insertion $p_r(AZBZ^*)$, let $\delta$ be a perfect matching of
$\{1,\bar1,\ldots,r,\bar r\}$ and set
\begin{equation}\label{eq:quaternionic-matching-types}
 A(\delta)=\Lambda(\delta,\delta_{(r)}),
 \qquad
 B(\delta)=\Lambda(\delta,\epsilon),
\end{equation}
using the two fixed matchings of
\eqref{eq:epsilon-matching}--\eqref{eq:delta-lambda-matching}.  Here it
is useful to regard $A(\delta)$ and $B(\delta)$ first as set
partitions of the $r$ Wick edges; attaching positive lengths to the
edges and summing within each block then produces the indices of the
power sums.

The associated M\"obius graph has one vertex, $r$ edges, and
$\ell(A)+\ell(B)$ faces.  Hence
\begin{equation}\label{eq:quaternionic-matching-Euler-characteristic}
 \chi(\delta)=1+\ell(A)+\ell(B)-r.
\end{equation}
After the rescaling in
\eqref{eq:standard-quaternionic-Gaussian}, its Wick weight is
\begin{equation}\label{eq:quaternionic-Wick-matching-weight}
 w_{\mathbb H}(\delta)=\frac14(-2)^{\chi(\delta)}.
\end{equation}
Finally, passage from a coproduct term $p_A\otimes p_B$ to our
differential monomial contributes
$\alpha^{\ell(A)-1}=2^{1-\ell(A)}$.  Thus the coefficient of a
single matching in the operator is
\begin{equation}\label{eq:quaternionic-operator-matching-weight}
 {
 c_{\mathbb H}(\delta)
 =(-1)^{1+\ell(A)+\ell(B)-r}2^{\ell(B)-r}.}
\end{equation}
This formula separates the two sources of powers of $2$: the Euler
characteristic in Wick's rule and the Jack normalization in the
coproduct-to-operator passage.

\subsection{Recovery of $\Delta_3(1/2)$}

There are $5!!=15$ matchings in degree three.  Up to simultaneous
relabelling of the three Wick edges, their types and Wick weights are
\begin{equation}\label{eq:quaternionic-degree-three-classes}
\begin{array}{c|c|c|c}
\text{multiplicity}&A&B&w_{\mathbb H}\\ \hline
1 &(123)&(1)(2)(3)&1\\
3 &(123)&(1)(23)&-1/2\\
3 &(1)(23)&(12)(3)&1\\
4 &(123)&(123)&1/4\\
3 &(1)(23)&(123)&-1/2\\
1 &(1)(2)(3)&(123)&1
\end{array}
\end{equation}
and the multiplicities sum to $15$.  Edge-labelling the six classes
and applying \eqref{eq:quaternionic-operator-matching-weight} first
recovers the whole top part $\Delta_3^{\mathrm{top}}(1/2)$.  The two
classes of Wick weight $-1/2$ give the remaining contribution
\begin{equation}\label{eq:quaternionic-Delta3-correction}
 -\frac34\sum_{i,j\geq1}(i+j-2)
 \left[\frac12(i,j\mid i+j)+(i+j\mid i,j)\right].
\end{equation}
For example, fixing the isolated edge and summing the two positive
lengths in the other block produces $i+j-2$ possibilities; the three
choices of the isolated edge account for the factor $3$.

At $\alpha=1/2$ the diagonal part of $\Delta'_3(\alpha)$ vanishes,
since $2\alpha-1=0$.  Consequently
\begin{proposition}\label{prop:quaternionic-Delta3}
The quaternionic Gaussian integral with insertion $p_3(AZBZ^*)$
produces
\begin{equation}
 {
 \mathcal H^{\mathbb H}_3
 =\Delta_3^{\mathrm{top}}(1/2)
  -\frac12\Delta'_3(1/2)
 =\Delta_3(1/2).}
\end{equation}
The monomial patterns of the top part are recovered first, while the
two signed classes give
\eqref{eq:quaternionic-Delta3-correction}.
\end{proposition}

\subsection{Recovery of $\Delta_4(1/2)$}

Degree four involves $7!!=105$ matchings.  The following table is a
complete classification up to simultaneous relabelling of the four
Wick edges.  The last column is the coefficient of one matching after
the operator normalization, not yet multiplied by the class size.
\begin{equation}\label{eq:quaternionic-degree-four-classes}
\begin{array}{r|c|c|r}
m&A&B&c_{\mathbb H}\\ \hline
1 &(1)(2)(3)(4)&(1234)&1/8\\
4 &(1)(2)(34)&(123)(4)&1/4\\
6 &(1)(2)(34)&(1234)&-1/8\\
2 &(1)(2)(34)&(13)(24)&1/4\\
4 &(1)(234)&(12)(3)(4)&1/2\\
4 &(1)(234)&(12)(34)&-1/4\\
8 &(1)(234)&(123)(4)&-1/4\\
16&(1)(234)&(1234)&1/8\\
2 &(12)(34)&(1)(23)(4)&1/2\\
4 &(12)(34)&(1)(234)&-1/4\\
5 &(12)(34)&(1234)&1/8\\
1 &(12)(34)&(13)(24)&-1/4\\
1 &(1234)&(1)(2)(3)(4)&1\\
6 &(1234)&(1)(2)(34)&-1/2\\
16&(1234)&(1)(234)&1/4\\
5 &(1234)&(12)(34)&1/4\\
20&(1234)&(1234)&-1/8
\end{array}
\end{equation}
The entries in the first column sum to $105$, and every entry in the
last column follows directly from
\eqref{eq:quaternionic-operator-matching-weight}.

Attach positive lengths $u_1,u_2,u_3,u_4$ to the Wick edges and write
$s=u_1+u_2+u_3+u_4$.  Regrouping the edge-labelled contributions of
\eqref{eq:quaternionic-degree-four-classes} by their resulting
differential monomial gives the five blocks
\begin{align}
&-\frac14\sum_{i,j,k,l\geq1}
 (i+j,k+l\mid i+l,j+k),
 \label{eq:quaternionic-Delta4-block-22}\\
&-\sum_{i,j,k\geq1}(i+j+k-3)
 \left[\frac14(i,j,k\mid i+j+k)
 +(i+j+k\mid i,j,k)\right],
 \label{eq:quaternionic-Delta4-block-31}\\
&-\sum_{i,j,k\geq1}(2i+j+k-4)
 (i+j,k\mid i+k,j),
 \label{eq:quaternionic-Delta4-block-mixed}\\
&-\frac12\sum_{i,j\geq1}F_{ij}(1/2)
 \left[\frac12(i,j\mid i+j)+(i+j\mid i,j)\right],
 \label{eq:quaternionic-Delta4-block-21}\\
&-\frac52\sum_{k\geq1}\binom{k-1}{3}(k\mid k).
 \label{eq:quaternionic-Delta4-block-11}
\end{align}
For orientation, the first line comes from the unique crossed class
$(12)(34)\mid(13)(24)$.  The last line comes from the $20$ matchings
of type $(1234)\mid(1234)$, each of coefficient $-1/8$; the binomial
coefficient counts compositions of $k$ into four positive parts.
The intermediate polynomial $F_{ij}$ is the one defined in
\eqref{eq:Fij-alpha}.

The coefficients in
\eqref{eq:quaternionic-Delta4-block-22}--
\eqref{eq:quaternionic-Delta4-block-11} are respectively the
specializations at $\alpha=1/2$ of the five lines of
$(\alpha-1)\Delta'_4(\alpha)$.  Indeed,
\[
 \alpha(\alpha-1)=-\frac14,
 \quad 2(\alpha-1)=-1,
 \quad 4\alpha(\alpha-1)=-1,
\]
and
\[
 (\alpha-1)(6\alpha^2-5\alpha+6)=-\frac52.
\]
We have therefore proved the quaternionic specialization of
Conjecture~\ref{conj:Delta4-explicit}.

\begin{proposition}\label{prop:quaternionic-Delta4}
The quaternionic Gaussian integral with insertion $p_4(AZBZ^*)$
produces
\begin{equation}\label{eq:quaternionic-Delta4}
 {
 \mathcal H^{\mathbb H}_4
 =\Delta_4^{\mathrm{top}}(1/2)
  -\frac12\Delta'_4(1/2)
 =\Delta_4(1/2).}
\end{equation}
\end{proposition}

As in the real case, the enumeration was also checked mechanically by
generating all perfect matchings and all positive edge-labellings of
total weight at most nine.  The comparison involves the same $49$
operator types for $\Delta'_3(1/2)$ and $154$ operator types for
$\Delta'_4(1/2)$, with no discrepancy.

Combining the real, complex, and quaternionic models, we obtain for
$r=2,3,4$
\begin{equation}\label{eq:three-matrix-specializations}
 {
 \mathcal H^{\mathbb R}_r=\Delta_r(2),\qquad
 \mathcal H^{\mathbb C}_r=\Delta_r(1),\qquad
 \mathcal H^{\mathbb H}_r=\Delta_r(1/2).}
\end{equation}
The crossed block
\[
 \alpha(\alpha-1)
 \sum_{i,j,k,l\geq1}(i+j,k+l\mid i+l,j+k)
\]
is especially revealing: its coefficient is $0$, $2$, and $-1/4$ in
the complex, real, and quaternionic cases, respectively.  It is absent
from the permutation model, positive in the real matching model, and
signed in the quaternionic M\"obius-graph model.
\section{Conclusions and open problems}
\label{sec:conclusions}

The stable-series construction developed in this paper places several
phenomena in a common framework.  In the classical specialization,
the series $\widehat f$ turn stable multiplication into differential
operators and recover the Ivanov--Kerov algebra together with shifted
symmetric functions.  The Goulden--Jackson product deforms this
construction to Jack symmetric functions.  Its diagonal action on the
Jack basis identifies the corresponding eigenvalues with shifted Jack
polynomials and gives, in particular, the shifted Pieri relation.

The differential-operator point of view is effective even when a
closed general formula is not yet available.  The operator
$\Delta_2(\alpha)$ is obtained exactly, while the calculations in
Section~\ref{sec:low-degree-jack-operators} lead to explicit candidate
formulas for $\Delta_3(\alpha)$ and $\Delta_4(\alpha)$.  We summarize
the logical status of the results used in the paper as follows.
\begin{center}
\begin{tabular}{p{0.61\textwidth}|p{0.25\textwidth}}
result&status\\ \hline
stable-series construction and classical differential operators
 & proved\\
diagonal action and shifted Jack Pieri formula
 & proved\\
formula for $\Delta_2(\alpha)$
 & proved\\
general formulas for $\Delta_3(\alpha)$ and $\Delta_4(\alpha)$
 & computer-assisted conjectures\\
real specialization of the proposed formulas at $\alpha=2$
 & verified by exhaustive Wick enumeration\\
quaternionic specialization of the proposed formulas at
$\alpha=1/2$
 & verified by exhaustive signed Wick enumeration
\end{tabular}
\end{center}
The last two statements mean that the matrix-integral operators agree
term by term with the displayed candidate formulas at the indicated
specializations.  They do not, by themselves, prove the formulas for
an indeterminate value of $\alpha$.

The matrix models nevertheless explain a structural feature which is
hard to see from the formulas alone.  Bipartite Wick matchings produce
the monomial patterns already visible in the permutation model.
Non-bipartite matchings produce the additional contractions.  In the
quaternionic model these contractions carry signs and powers of two
controlled by the Euler characteristic of the associated M\"obius
graph.  A particularly transparent example is
\begin{equation}\label{eq:conclusion-crossed-block}
 \alpha(\alpha-1)
 \sum_{i,j,k,l\geq1}
 (i+j,k+l\mid i+l,j+k).
\end{equation}
Its coefficient is respectively $0$, $2$, and $-1/4$ for
$\alpha=1,2,1/2$.  Thus the same contraction is absent from the
permutation model, counted positively in the real model, and signed
in the quaternionic model.

Several natural questions remain open.

First, the candidate formulas for $\Delta_3(\alpha)$ and
$\Delta_4(\alpha)$ should be derived by a complete normal-ordering
argument, or by a conceptual construction which makes their
polynomial dependence on $\alpha$ manifest.  Such a proof should also
clarify how to construct $\Delta_\mu(\alpha)$ directly for an
arbitrary partition $\mu$.

Second, the elementary-symmetric operators deserve a separate
explanation.  Their deformed form is strikingly close to the classical
Nazarov--Sklyanin form, whereas the power-sum operators acquire the
non-permutation corrections exhibited above.  Understanding this
contrast may provide a more intrinsic construction of the commuting
Jack operators.

Third, it would be desirable to interpolate the real, complex, and
quaternionic contraction rules without first knowing the answer as a
polynomial in $\alpha$.  The appearance of perfect matchings suggests
a possible relation with non-orientability statistics and with the
Matching--Jack problem.  We do not pursue that question here: a
reproducible comparison requires a precise deletion algorithm and a
fully specified measure of non-orientability, and lies beyond the
verified calculations of the present paper.

The main outcome is therefore not a closed formula for every Jack
operator, but a framework in which stability, shifted symmetry,
differential operators, and the three classical matrix models become
different manifestations of the same construction.

\section*{Acknowledgments}
The author used ChatGPT (OpenAI) as an interactive aid in organizing
the manuscript, editing the English text, and checking
computer-algebra calculations.  All mathematical statements and
conclusions remain the responsibility of the author.


\bibliographystyle{plain}

%

\bibliography{references}

\end{document}